\documentclass[11pt,reqno]{amsart}
\usepackage{amsmath,amssymb,mathrsfs}
\usepackage{graphicx,cite}
\usepackage{color}

\usepackage{epstopdf}
\ifpdf
  \DeclareGraphicsExtensions{.eps,.pdf,.png,.jpg}
\else
  \DeclareGraphicsExtensions{.eps}
\fi

\newtheorem{theorem}{Theorem}[section]
\newtheorem{corollary}[theorem]{Corollary}
\newtheorem{lemma}[theorem]{Lemma}

\newtheorem{example}{Example}
\newtheorem{assumption}{Assumption}

\numberwithin{equation}{section}

\allowdisplaybreaks[4]

\begin{document}

\title[An inverse random source problem]{An inverse random source
problem for the one-dimensional Helmholtz equation with attenuation}

\author{Peijun Li}
\address{Department of Mathematics, Purdue University, West Lafayette, Indiana 47907, USA}
\email{lipeijun@math.purdue.edu}

\author{Xu Wang}
\address{Department of Mathematics, Purdue University, West Lafayette, Indiana 47907, USA}
\email{wang4191@purdue.edu}

\thanks{The research is supported in by part the NSF grant DMS-1912704.}



\begin{abstract}
This paper is concerned with an inverse random source problem for the
one-dimensional stochastic Helmholtz equation with attenuation. The
source is assumed to be a microlocally isotropic Gaussian random field with its
covariance operator being a classical pseudo-differential operator. The random
sources under consideration are equivalent to the generalized fractional Gaussian
random fields which include rough fields and can be even rougher than the white
noise, and hence should be interpreted as distributions. The well-posedness of
the direct source scattering problem is established in the distribution
sense. The micro-correlation strength of the random source, which appears to be
the strength in the principal symbol of the covariance operator, is proved to be
uniquely determined by the wave field in an open measurement set. Numerical
experiments are presented for the white noise model to demonstrate
the validity and effectiveness of the proposed method. 
\end{abstract}

\maketitle

\section{Introduction}

Inverse source problems for wave propagation aim to determine the unknown
sources by using supplementary information of the wave field. They arise
naturally and have significant applications in diverse fields of science, which
include particularly the area of medical and biomedical imaging such as
magnetoencephalography \cite{ABF02,FKM04}, optical molecular imaging
\cite{BT07}, and fluorescence tomography \cite{CYZB14}. Motivated by these
applications, inverse source scattering problems have been extensively
investigated, and many mathematical and numerical results are available
\cite{BLZ20,BLT10,I90,IL18}.

Recently, to characterize more precisely the uncertainties in unpredictable
systems with incomplete knowledge, random sources are taken into consideration
in mathematical modeling \cite{BCL16,BCL18,D79}. As is well known, classical
inverse problems are already rather difficult to solve due to the nonlinearity
and ill-posedness. Inverse problems with random sources would be more
challenging since the ill-posedness is severer compared to their deterministic
counterparts: (1) the random source, in some cases, is too rough to exist
pointwisely and should be interpreted as distributions instead; (2) the wave
field generated by the random source is also a random field. Random fields are
determined by their statistics such as the mean and covariance functions. As a
result, only statistics of the random source may be reconstructed based on the
statistics of the wave field. It is worth pointing out that the statistics of
the random source which can be determined and the statistics of the wave field
which can be used as proper measurement data depend heavily on the form of the
random source, which makes it hard to solve inverse random source problems.

In this paper, we consider the one-dimensional stochastic Helmholtz equation
with attenuation 
\begin{equation}\label{eq:model}
u''(x)+(k^2+{\rm i}k\sigma) u(x)=f(x),\quad x\in\mathbb{R},
\end{equation}
where $k>0$ is the wave number, the attenuation coefficient $\sigma>0$ describes
the electrical conductivity of the medium, $u$ denotes the scattered field, and
$f$ represents the electric current density and is assumed to be a random
field supported in $D=(0,1)$. In the one-dimensional case, the outgoing radiation
condition imposed on $u$ is equivalent to the following boundary conditions:  
\[
u'(0)+{\rm i}\kappa u(0)=0, \quad
u'(1)+{\rm i}\kappa u(1)=0, 
\]
which accounts for the left-going wave at $x=0$ and the right-going wave at
$x=1$, respectively. Here $\kappa$ satisfies $\kappa^2=k^2+{\rm i}\sigma k$. 

There has been much work on the study of inverse random source problems. When
the source takes the form $f=g+h\dot{W}$, where $\dot{W}$ is the spatial white
noise, $g$ and $h$ are smooth and compactly supported functions, the random
source has independent increments. As a result, the It\^o isometry can be used
to derive reconstruction formulas which connect the statistics of the random
source to those of the wave field, and the functions $g$ and $h$ can be
determined based on the measurement data at multiple frequencies. We refer to
\cite{BCL16,BCLZ14,L11} for the study on the stochastic Helmholtz equation
without attenuation and to \cite{BCL18} for the study on the stochastic elastic
wave equation. 

More generally, another important class of random sources, known as the
microlocally isotropic Gaussian random fields, is considered in
\cite{LPS08,LHL,LL19,LLM,LW,LW2}. The covariance operators of the random fields
are assumed to be pseudo-differential operators with principal symbol
$\mu(x)|\xi|^{-m}$, where the nonnegative function $\mu\in C_0^{\infty}(D)$ is
called the micro-correlation strength of the random source and is the statistics
to be determined. It is shown in \cite{LW} that the microlocally isotropic
Gaussian random field is equivalent to the generalized fractional Gaussian
random field in the form
\[
f=\sqrt{\mu}(-\Delta)^{-\frac m4}\dot{W},
\]
which is a distribution in $W^{\frac{m-d}2-\epsilon,p}(\mathbb R^d)$ for
$m\in(-\infty,d]$ (cf. Lemma \ref{lm:iso}) and apparently degenerates to the
white noise if $m=0$. In this case, the increments of the random source are not
independent if $m\neq0$, and thus the It\^o isometry is not applicable any
more. Instead, the microlocal analysis for large frequencies is applied to
reconstruct the micro-correlation strength $\mu$ involved in the principal
symbol of the covariance operator of $f$. In \cite{LW}, the $d$-dimensional
Helmholtz equation with attenuation is studied with $d=2,3$, $p\in(\frac
d2,2]$ and $m\in(d(\frac2p+1)-2,d]$. We refer to \cite{LHL,LL19} for the
study on the Helmholtz equation without attenuation and the elastic wave
equation, to \cite{LLM} for the study on the Schr\"odinger equation, and to
\cite{LW2} for the study on Maxwell's equations. In all of the existing results,
the random source under consideration is smoother than the white noise, i.e.,
$m>0$, due to the singularity of Green's functions of the considered models. 

In this work, we consider the one-dimensional stochastic Helmholtz equation
(\ref{eq:model}) with attenuation, where $f$ is assumed to be a microlocally
isotropic Gaussian random field with $m\in(-\frac2q,1]$ and $q\in(1,\infty)$.
We point out that such a random source model includes the white noise case with
$m=0$ and is even allowed to be rougher than the white noise for
$m\in(-\frac2q,0)$. The direct scattering problem is shown to be well-posed in
the distribution sense and has a unique solution $u\in
W^{\gamma,q}_{loc}(\mathbb R)$ with $\gamma\in(\frac{1-m}2,\frac12+\frac1q)$. 
For the inverse scattering problem, we prove that the strength $\mu$ of the
random source is uniquely determined by the high frequency limit of the energy
of the wave field $u$ on an bounded measurement interval $U\subset\mathbb
R\backslash\overline{D}$. In particular, for the white noise case, the
measurement data at a single frequency is enough to uniquely determine the
strength $\mu$ by utilizing the It\^o isometry. Numerical experiments are
presented for the white noise model to demonstrate the validity and
effectiveness of the proposed method.

The paper is organized as follows. In Section \ref{dsp}, the microlocally
isotropic random source is introduced. The well-posedness of the direct
scattering problem in the distribution sense is given based on the regularity of
the fundamental solution. Section \ref{isp} concerns the inverse scattering
problem. The uniqueness is addressed for the reconstruction of the strength of
the random source. As a special case of the microlocally isotropic random
source, the white noise model is studied in Section \ref{wn}. Numerical
experiments are presented in Section \ref{ne} to demonstrate the effectiveness
of the proposed method. The paper is concluded with some general remarks in
Section \ref{c}.

\section{Direct scattering problem}\label{dsp}

In this section, we introduce the model of the random source and present the
well-posedness and stability of the solution for the direct scattering problem. 

\subsection{Random sources}

The source $f$ is assumed to be a microlocally isotropic Gaussian random
field which satisfies the following conditions with dimension $d=1$.

\begin{assumption}\label{as:f}
Let $f$ be a real-valued centered microlocally isotropic Gaussian random field
of order $-m$ compactly supported in $D\subset\mathbb R^d$, i.e., the
covariance operator of $f$ is a pseudo-differential operator whose principal
symbol has the form $\mu(x)|\xi|^{-m}$ with the micro-correlation strength
$\mu\in C_0^{\infty}(D)$ and $\mu\ge0$.
\end{assumption}

It is shown in \cite[Proposition 2.5]{LW} that the generalized Gaussian
random field
\[
f(x)=\sqrt{\mu(x)}(-\Delta)^{-\frac m4}\dot{W}
\]
satisfies Assumption \ref{as:f} with order
$-m$, where $\dot{W}$ is the white noise and $(-\Delta)^{-\frac m4}$ is a
fractional Laplacian. Consequently, the regularity of random fields satisfying
Assumption \ref{as:f} can be obtained by investigating the regularity of the
generalized Gaussian random fields, which is stated in the following lemma (cf.
\cite{LW}).

\begin{lemma}\label{lm:iso}
Let $f$ be a microlocally isotropic Gaussian random field of order $-m$ compactly supported in $D\subset\mathbb R^d$.
\begin{itemize}
\item[(i)] If $m\in(d,d+2)$, then $f\in C^{\alpha}(D)$ almost surely for all $\alpha\in(0,\frac{m-d}2)$.
\item[(ii)] If $m\in(-\infty,d]$, then $f\in W^{\frac{m-d}2-\epsilon,p}(D)$ almost surely for all $\epsilon>0$ and $p\in(1,\infty)$.
\end{itemize}
\end{lemma}

Let $\mathcal{D}$ be the space $C_0^{\infty}(\mathbb R^d)$ equipped with a locally convex topology, and $\mathcal{D}'$ be its dual space. Based on Lemma \ref{lm:iso}, if $m\le d$, the random source $f$ should be interpreted as a distribution in $\mathcal{D}'$. Its mean value function, denoted by $E_f$, and covariance operator, denoted by $Q_f$, are defined as follows:
\begin{eqnarray*}
\langle E_f,\varphi\rangle:&=&\mathbb E\langle f,\varphi\rangle\quad \forall~\varphi\in\mathcal{D},\\
\langle\varphi,Q_f\psi\rangle:&=&\mathbb E[\langle f,\varphi\rangle\langle f,\psi\rangle]\quad\forall~\varphi,\psi\in\mathcal{D},
\end{eqnarray*}
where $\langle\cdot,\cdot\rangle$ denotes the dual product. 
According to the Schwartz kernel theorem (cf. \cite[Theorem 5.2.1]{H03}), there exists a unique kernel $K_f$ for $Q_f$ such that 
\begin{equation}\label{eq:Qf}
\langle\varphi,Q_f\psi\rangle=\int_{\mathbb R^d}\int_{\mathbb R^d}K_f(x,y)\varphi(x)\psi(y)dxdy.
\end{equation}

If $f$ satisfies Assumption \ref{as:f}, then its covariance operator $Q_f$ is a
pseudo-differential operator with the principal symbol given by
$\mu(x)|\xi|^{-m}$, and hence (cf. \cite{H07})
\[
(Q_f\psi)(x)=\frac1{(2\pi)^d}\int_{\mathbb R^d}e^{{\rm i}x\cdot\xi}c(x,\xi)\hat\psi(\xi)d\xi,
\]
where $c(x,\xi)$ is the symbol of $Q_f$ with the leading term $\mu(x)|\xi|^{-m}$
and 
\[
\hat\psi(\xi)=\mathcal{F}[\psi](\xi)=\int_{\mathbb R^d}e^{-{\rm i}x\cdot\xi}\psi(x)dx
\]
is the Fourier transform of $\psi$.
It then holds
\begin{eqnarray*}
\langle\varphi,Q_f\psi\rangle
&=&\int_{{\mathbb{R}}^d}\varphi(x)\left[\frac1{(2\pi)^{d}}\int_{{\mathbb{R}}^d}e^
{{\rm i} x\cdot \xi}c(x,\xi)\hat\psi(\xi)d\xi\right]dx\\
&=&\frac1{(2\pi)^{d}}\int_{{\mathbb{R}}^d}\varphi(x)\int_{{\mathbb{R}}^d}e^{{
\mathbf{i}} x\cdot \xi}c(x,\xi)\left[\int_{{\mathbb{R}}^d}e^{-{\rm i}
y\cdot\xi}\psi(y)dy\right]d\xi dx\\
&=&\int_{{\mathbb{R}}^d}\int_{{\mathbb{R}}^d}\left[\frac1{(2\pi)^d}\int_{{\mathbb
{R}}^d}e^{{\rm i}(x-y)\cdot\xi}c(x,\xi)d\xi\right]\varphi(x)\psi(y)dxdy.
\end{eqnarray*}
Comparing the above equation with (\ref{eq:Qf}), we get that the kernel $K_f$ is
an oscillatory integral of the form
\begin{equation}\label{eq:Kf}
K_f(x,y)=\frac1{(2\pi)^d}\int_{{\mathbb{R}}^d}e^{{\rm i}(x-y)\cdot\xi}c(x,
\xi)d\xi,
\end{equation}
which is determined by the symbol $c(x,\xi)$.

\subsection{The fundamental solution}

Define the complex wave number $\kappa$ such that $\kappa^2=k^2+{\rm i}k\sigma$,
whose real and imaginary parts $\kappa_{\rm r}$ and $\kappa_{\rm i}$ satisfy
\[
\kappa_{\rm
r}=\left(\frac{\sqrt{k^4+k^2\sigma^2}+k^2}2\right)^{\frac12 } , \quad
\kappa_{\rm i}=\left(\frac{\sqrt{k^4+k^2\sigma^2}-k^2}2\right)^{\frac12}.
\]
It is easy to verify that 
\begin{equation}\label{eq:kappa}
 \lim_{k\to\infty}\frac{\kappa_{\rm r}}{k}=1,\quad \lim_{k\to\infty}\kappa_{\rm
i}=\frac{\sigma}{2}. 
\end{equation}

Before showing the well-posedness of the solution for the random equation
(\ref{eq:model}), we recall that the equation 
\begin{equation*}\label{eq:funda}
(\partial_{xx}+\kappa^2)\Phi_\kappa(x,y)=-\delta(x-y),\quad x,y\in\mathbb{R}
\end{equation*}
admits a unique solution 
\[
\Phi_\kappa(x,y)=\frac{{\rm i}}{2\kappa}e^{{\rm i} \kappa|x-y|},
\]
which is the fundamental solution for the one-dimensional Helmholtz equation. 

For any $n\in\mathbb N$, denote by $W^{n,p}(\mathcal{O})$ the Sobolev space equipped with the norm
\[
\|v\|_{W^{n,p}(\mathcal{O})}:=\left(\sum_{0\le\alpha\le
n}\|\partial^{\alpha}v\|_{L^p(\mathcal{O})}^p\right)^{\frac1p}.
\] 
Let $W^{n,p}_0(\mathcal{O})$ be the closure of $C_0^{\infty}(\mathcal{O})$ in
$W^{n,p}(\mathcal{O})$ and $W^{-n,q}(\mathcal{O})=(W^{n,p}_0(\mathcal{O}))'$ be 
the dual space of $W^{n,p}_0(\mathcal{O})$ with $\frac1p+\frac1q=1$. We refer to
\cite{AF03} for more details on these Sobolev spaces. 

The fundamental solution $\Phi_\kappa$ has the following regularity property.

\begin{lemma}\label{lm:Phik}
For any given $x\in\mathbb{R}$ and $p\in(1,\infty)$, it holds 
$\Phi_\kappa(x,\cdot)\in W_{loc}^{1,p}(\mathbb R).$
\end{lemma}

\begin{proof}
Let $\mathcal{O}\subset\mathbb{R}$ be any bounded interval with a finite
Lebesgue measure which is denoted by $C_{\mathcal{O}}$. It
suffices to show that $\Phi_\kappa(x,\cdot),\partial\Phi_\kappa(x,\cdot)\in
L^p(\mathcal{O})$. A simple calculation gives
\begin{eqnarray*}
\|\Phi_\kappa(x,\cdot)\|_{L^p(\mathcal{O})}^p&=&\int_{\mathcal{O}}\left|\frac{{\rm i}}{2\kappa}e^{{\rm i} \kappa|x-y|}\right|^pdy=\int_{\mathcal{O}}\left(\frac1{2|\kappa|}\right)^pe^{-p\kappa_{\rm i}|x-y|}dy\\
&\le&(2|\kappa|)^{-p}C_{\mathcal{O}}.
\end{eqnarray*}
Since the classical partial derivative of $\Phi_\kappa(x,y)$ with respect to $y$ exists, we have
\begin{equation*}
\partial_y\Phi_\kappa(x,y)=\frac{x-y}{2|x-y|}e^{{\rm i} \kappa|x-y|},\quad y\neq x.
\end{equation*}
It is clear to note
\[
\|\partial\Phi_\kappa(x,\cdot)\|_{L^p(\mathcal{O})}^p=\frac1{2^p}\int_{\mathcal{
O}}e^{-p\kappa_{\rm i}|x-y|}dy\le2^{-p}C_{\mathcal{O}},
\]
which completes the proof.
\end{proof}

\subsection{Well-posedness and regularity}

Based on the fundamental solution $\Phi_\kappa$, the volume potential 
\begin{equation}\label{eq:Hk}
(H_\kappa v)(x):=-\int_\mathbb{R}\Phi_\kappa(x,y)v(y)dy
\end{equation}
defines a mollifier $H_\kappa $.

\begin{lemma}\label{lm:bounded}
Let $\mathcal{O},\mathcal{V}\subset\mathbb{R}$ be any two bounded intervals. The
operator $H_\kappa:H^{-\beta}(\mathcal{O})\to H^\beta(\mathcal{V})$ is bounded
for $\beta\in(0,1]$.
\end{lemma}

\begin{proof}
It follows from \cite[Theorem 8.1]{CK13} that $H_\kappa$ is bounded from
$C(\mathcal{O})$ to $C^2(\mathcal{V})$ with respect to the norms 
\[
\|v\|_{C(\mathcal{O})}:=\sup_{x\in\mathbb R}|v(x)|\quad \forall~v\in C(\mathcal{O})
\] 
and 
\[
\|v\|_{C^2(\mathcal{V})}:=\sum_{m=0}^2\sup_{x\in\mathbb R}|v^{(m)}(x)|\quad\forall~v\in C^2(\mathcal{V}).
\]
Define spaces $X:=C(\mathcal{O})$ and $Y:=C^2(\mathcal{V})$ with the scalar
products given by 
\[
(g_1,g_2)_X:=(\tilde g_1,\tilde
g_2)_{H^{\beta-2}(\mathbb{R})}\quad\forall\,g_1, g_2\in X
\]
and
\[
(h_1,h_2)_Y:=(\tilde h_1,\tilde h_2)_{H^{\beta}(\mathbb{R})}\quad \forall\,
h_1, h_2\in Y,
\]
respectively, where $\tilde g_i$ and $\tilde h_i$ are the zero extensions of
$g_i$ and $h_i$  in $\mathbb R\setminus\overline{\mathcal{O}}$ and $\mathbb
R\setminus\overline{\mathcal{V}}$, respectively. It is easy to verify that 
the products defined above satisfy 
\begin{eqnarray*}
(g_1,g_2)_X &=& (J^{\beta-2}\tilde g_1,J^{\beta-2}\tilde
g_2)_{L^2(\mathbb{R})}=\int_{\mathbb R}(1+\xi^2)^{\beta-2}\hat{\tilde
g}_1(\xi)\overline{\hat{\tilde g}_2(\xi)}d\xi\\
&\lesssim& \|\tilde g_1\|_{L^2(\mathbb R)}\|\tilde g_2\|_{L^2(\mathbb R)}
\lesssim\|g_1\|_{C(\mathcal{O})}\|g_2\|_{C(\mathcal{O})}
\end{eqnarray*}
and
\begin{eqnarray*}
(h_1,h_2)_Y&=&(J^{\beta}\tilde h_1,J^{\beta}\tilde h_2)_{L^2(\mathbb{R})}=\int_{\mathbb R}(1+\xi^2)^{\beta}\hat{\tilde
h}_1(\xi)\overline{\hat{\tilde h}_2(\xi)}d\xi\\
&\lesssim&\|\tilde h_1\|_{H^\beta(\mathbb R)}\|\tilde h_2\|_{H^\beta(\mathbb R)}
\lesssim\|h_1\|_{C^2(\mathcal{V})}\|h_2\|_{C^2(\mathcal{V})}.
\end{eqnarray*}

We claim that there exists a bounded operator $V:Y\to X$ defined by 
\[
V:=(I-\partial_{xx})\overline{H_\kappa}(I-\partial_{xx}),
\]
where 
\[
(\overline{H_\kappa}v)(x):=-\int_{\mathbb R}\overline{\Phi_\kappa(x,y)}v(y)dy
\]
and $\overline{H_\kappa}v=\overline{H_\kappa v}$ for any real valued function $v$, 
such that
\[
(H_\kappa g,h)_Y=(g,Vh)_X\quad\forall\,g\in X, h\in Y.
\]
In fact, for any $h\in Y$, 
\begin{eqnarray*}
\|Vh\|_{C(\mathcal{O})}&=&\|(I-\partial_{xx})\overline{H_\kappa} (I-\partial_{xx})h\|_{C(\mathcal{O})}
\lesssim\|H_\kappa (I-\partial_{xx})h\|_{C^2(\mathcal{O})}\\
&\lesssim& \|(I-\partial_{xx})h\|_{C(\mathcal{V})}\lesssim\|h\|_{C^2(\mathcal{V})}.
\end{eqnarray*}
Furthermore,
\begin{eqnarray*}
(H_\kappa g,h)_Y&=&(J^\beta H_\kappa \tilde g,J^\beta\tilde h)_{L^2(\mathbb R)}=(H_\kappa \tilde
g,J^{2\beta}\tilde h)_{L^2(\mathbb R)}\\
&=&(\widehat{H_\kappa \tilde g},\widehat{J^\beta\tilde h})_{L^2(\mathbb R)}
=\int_{\mathbb R}\hat\Phi_\kappa(\xi)\hat{\tilde g}(\xi)(1+\xi^2)^\beta\overline{\hat{\tilde h}(\xi)}d\xi\\
&=&\int_{\mathbb R}\hat{\tilde
g}(\xi)(1+\xi^2)^{\beta-2}\overline{\left[
(1+\xi^2)\overline{\hat\Phi_\kappa(\xi)}(1+\xi^2)\hat{\tilde h}(\xi)\right]}d\xi\\
&=&\int_{\mathbb R}\hat{\tilde
g}(\xi)(1+\xi^2)^{\beta-2}\overline{\widehat{V\tilde h}(\xi)}d\xi
=(J^{\beta-2}\tilde g,J^{\beta-2}V\tilde h)_{L^2(\mathbb R)}\\
&=&(g,Vh)_X,
\end{eqnarray*}
where $\hat\Phi_\kappa$ is the Fourier transform of $\Phi_\kappa(x,y)$ with respect to
$x-y$ and satisfies $-\xi^2\hat\Phi_\kappa(\xi)+\kappa^2\hat\Phi_\kappa(\xi)=-1$. The claim
is proved. 

It follows from the claim and  \cite[Theorem 3.5]{CK13} that $H_\kappa :X\to Y$ is
bounded with respect to the norms induced by the scalar products on $X$ and
$Y$. More precisely, we have 
\begin{equation}\label{eq:Hkineq}
\|H_\kappa g\|_Y=\|H_\kappa g\|_{H^\beta(\mathcal{V})}\lesssim\|g\|_X=\|g\|_{H^{\beta-2}(\mathcal{O})}\le\|g\|_{H^{-\beta}(\mathcal{O})}
\end{equation}
for any $g\in X$ and $\beta\le1$. It then suffices to show that (\ref{eq:Hkineq}) also holds for any $g\in H^{-\beta}(\mathcal{O})$.
Noting that the subspace $C_0^\infty(\mathcal{O})\subset X$ is dense in $L^2(\mathcal{O})$ (cf. \cite[Section 2.30]{AF03}) and 
$
H^{-1}(\mathcal{O})=\overline{L^2(\mathcal{O})}^{\|\cdot\|_{H^{-1}(\mathcal{O})}}
$
(cf. \cite[Section 3.13]{AF03}), we get that (\ref{eq:Hkineq}) holds for any $g\in
H^{-1}(\mathcal{O})$, and hence for any $g\in H^{-\beta}(\mathcal{O})$ since $H^{-\beta}(\mathcal{O})\subset
H^{-1}(\mathcal{O})$. 
\end{proof}

Now we are able to show the well-posedness of (\ref{eq:model}) in the distribution sense.

\begin{theorem}\label{tm:wellposed}
Let $q\in(1,\infty)$ and Assumption \ref{as:f} hold with $m\in(-\frac2q,1]$. The scattering problem (\ref{eq:model}) has a unique solution 
\[
u(x)=-\int_{D}\Phi_\kappa(x,y)f(y)dy
\]
in the distribution sense, and $u\in W^{\gamma,q}_{loc}(\mathbb R)$ almost surely with $\gamma\in(\frac{1-m}2,\frac12+\frac1q)$.
\end{theorem}

\begin{proof}
We first show that the volume potential 
\[
u(x)=-\int_{D}\Phi_\kappa(x,y)f(y)dy=(H_\kappa f)(x)
\]
is well-defined in $W^{\gamma,q}_{loc}(\mathbb R)$, i.e., $u\in
W^{\gamma,q}(K)$ for any compact subset $K\subset\mathbb R$. It follows from
the Kondrachov embedding theorem that the following embeddings
\[
W^{-\gamma,p}(D)\hookrightarrow H^{-\beta}(D),\quad H^\beta(K)\hookrightarrow W^{\gamma,q}(K)
\]
with $\beta=1$ and $\frac1p+\frac1q=1$ are compact. Hence,
$H_\kappa:W^{-\gamma,p}(D)\to W^{\gamma,q}(K)$ is bounded based on Lemma
\ref{lm:bounded}. By Lemma \ref{lm:iso}, it is clear to note that $f\in
W^{\frac{m-1}2-\epsilon,p}(D)\subset W^{-\gamma,p}(D)$. As a result, $u=H_\kappa
f\in W^{\gamma,q}(K)$.

Next, we prove that $u=H_\kappa f$ is a solution to (\ref{eq:model}) in the distribution sense. For any test function $v\in W^{\gamma,q}(\mathbb R)$, it holds
\begin{eqnarray*}
&&\langle u''+\kappa^2u,v\rangle=-\langle u',
v' \rangle+\kappa^2\langle u,v\rangle\\
&=&\int_{\mathbb{R}}\partial_x \Big[\int_{D}\Phi_\kappa(x,
y)f(y)dy \Big]
v'(x)dx -\kappa^2\int_{\mathbb{R}}\Big[\int_{D}\Phi_\kappa(x,
y)f(y)dy\Big] v(x)dx\\
&=&-\int_{D}\int_{\mathbb{R}}\partial_{xx}\Phi_\kappa(x,
y)v(x)f(y)dxdy -\kappa^2\int_{\mathbb{R}}\Big[\int_{D}
\Phi_\kappa(x,y)f(y)dy\Big] v(x)dx\\
&=&\int_{D}\int_{\mathbb{R}}\left(\kappa^2\Phi_\kappa(x,
y)+\delta(x-y)\right)v(x)f(y)dxdy\\
& &-\kappa^2\int_{\mathbb{R}}\Big[\int_{D}\Phi_\kappa(x,y)f(y)dy\Big] v(x)dx\\
&=&\langle f,v\rangle.
\end{eqnarray*}

The uniqueness of the solution of (\ref{eq:model}) can be proved by
showing that (\ref{eq:model}) has only the zero solution if $f\equiv0$. Let $ u^0$ be any
solution of (\ref{eq:model}) with $f\equiv0$ in the distribution sense. Then $ u^0$ satisfies
\[
(u^0)''+\kappa^2 u^0=0
\]
in the distribution sense. Denote $B_r=(-r,r)$. It is shown in Lemma \ref{lm:Phik} that $\Phi_\kappa(x,\cdot)\in W^{1,p'}(B_r)\hookrightarrow W^{\gamma,q}(B_r)$ for some $p'>1$ satisfying $\frac1{p'}-(1-\gamma)<\frac1q$.
It then indicates that $1_{B_r} \Phi_\kappa(x,\cdot)\in W^{\gamma,q}(\mathbb R)$, where $1_{B_r}$ denotes the characteristic function.
Hence, we get
\begin{equation}\label{eq:u0}
\int_{\mathbb R} \Phi_\kappa(x,z)\left[(u^0)''(z)+\kappa^2 u^0(z)\right]dz
=0.
\end{equation}
Define the operator $T$ by
\[
(T\psi)(x):=\int_{B_r}\Phi_\kappa(x,z)[\psi''(z)+\kappa^2\psi(z)]dz\quad\forall
\,\psi\in\mathcal{D}.
\] 
Following the similar arguments as those in the proof of \cite[Lemma 4.3]{LHL}
and using the integration by parts, we obtain
\begin{eqnarray*}
(T\psi)(x)=-\psi(x)+\left[\Phi_\kappa(x,z)\psi'(z)-\partial_z\Phi_\kappa(x,z)\psi(z)\right]\big|_{z=-r}^r.
\end{eqnarray*}
Then (\ref{eq:u0}) leads to
\begin{equation*}
-u^0(x)+\lim_{r\to\infty}\left[\Phi_\kappa(x,z)(u^0)'(z)-\partial_z\Phi_\kappa(x,z)u^0(z)\right]\big|_{z=-r}^r
=0.
\end{equation*}
Applying the radiation condition, we get
$u^0\equiv0$, which completes the proof.
\end{proof}

\section{Inverse scattering problem}\label{isp}

By Theorem \ref{tm:wellposed}, the solution of
(\ref{eq:model}) has the form
\begin{equation}\label{eq:solu}
u(x)=\frac1{2{\rm i}\kappa}\int_{D}e^{{\rm i} \kappa|x-y|}f(y)dy.
\end{equation}
We show that the micro-correlation strength $\mu$ is uniquely determined by the
variance of the solution $u$.

\begin{theorem}\label{tm:inverse}
Let $f$ be a random source satisfying Assumption \ref{as:f} and $U\subset\mathbb
R\setminus\overline{D}$ be a bounded open interval. Then for any$x\in U$, 
\[
\lim_{k\to\infty}4k^{m+2}{\mathbb{E}}|u(x)|^2=\int_{D}e^{-\sigma|x-y|}\mu(y)dy=:T(x).
\]
\end{theorem}

\begin{proof}
Since $U$ and $D$ are disjoint, we first consider the case $x>y$ for any $x\in
U$ and $y\in D$. Using (\ref{eq:solu}) and the fact that $f$ is compactly
supported in $D$, we have for any $x\in U$ that 
\begin{eqnarray*}
{\mathbb{E}}|u(x)|^2&=&\frac1{4|\kappa|^2}\int_{D}\int_{D}e^{{\rm i}\kappa|x-y|-{\rm i}\bar\kappa|x-z|}
{\mathbb{E}}[f(y)f(z)]dydz\\
&=&\frac1{4|\kappa|^2}\int_{D}\int_{D
}e^{{\rm i}\kappa(x-y)-{\rm i}\bar\kappa(x-z)}K_f(y,
z)\theta(x)dydz\\
&=&\frac1{4|\kappa|^2}\int_{\mathbb R}\int_{\mathbb R
}e^{\kappa_{\rm i}(y+z-2x)-{\rm i}\kappa_{\rm r}(y-z)}C_1(y ,
z,x)dydz,
\end{eqnarray*}
where $\theta\in C_0^{\infty}(\mathbb R)$ such that $\theta|_{U}\equiv1$ and
supp$(\theta)\subset{\mathbb{R}}\backslash\overline{D}$, 
\[
C_1(y,z,x):=K_f(y,z)\theta(x)=\frac1{2\pi}\int_{{\mathbb{R}}}e^{{\rm
i}(y-z)\xi}c_1(y,x,\xi)d\xi.
\]
Here $c_1(y,x,\xi):=c(y,\xi)\theta(x)$ and $c(y,\xi)$ is the symbol of the covariance operator of $f$. Then according to Assumption \ref{as:f}, the principal symbol of $c_1$ has the form
\[
c_1^p(y,x,\xi)=\mu(y)\theta(x)|\xi|^{-m}.
\]

First we define an invertible
transformation $\tau:{\mathbb{R}}^3\to{\mathbb{R}}^3$ by $
\tau(y,z,x)=(g,h,x)$, where 
\begin{eqnarray*}
g=y-z,\quad h=y+z.
\end{eqnarray*}
It follows from a straightforward calculation that  
\[
{\mathbb{E}}|u(x)|^2=\frac1{4|\kappa|^2}\int_{\mathbb R}\int_{\mathbb R
}e^{\kappa_{\rm i}h-2\kappa_{\rm i}x-{\rm i}\kappa_{\rm r}g}C_2(g,h,x)dgdh,
\]
where 
\begin{eqnarray*}
C_2(g,h,x)&=&C_1(\tau^{-1}(g,h,x))|\det\left((\tau^{-1})'(g,h,
x)\right)|\\
&=&\frac12C_1\left(\frac{g+h}2,\frac{h-g}2,x\right)\\
&=&\frac1{4\pi}\int_{\mathbb R}e^{{\rm i}g\xi}c_1\Big(\frac{g+h}2,x,\xi\Big)d\xi\\
&=&\frac1{4\pi}\int_{\mathbb R}e^{{\rm i}g\xi}c_2\Big(h,x,\xi\Big)d\xi.
\end{eqnarray*}
Here in the last step, we have used the following asymptotic expansion of
symbols (cf. \cite[Lemma 18.2.1]{H07}):
\begin{eqnarray*}
c_2(h,x,\xi)&= &e^{-{\rm i}\left\langle
D_g,D_{\xi}\right\rangle}c_1\Big(\frac{g+h}2,x,\xi\Big)\Big|_{g=0
}\\
&=&\sum_{j=0}^{\infty}\frac{\langle-{\rm i}
D_g,D_\xi\rangle^j}{j!}c_1\Big(\frac{g+h}2,x,\xi\Big)\Big|_{g=0}.
\end{eqnarray*}
Therefore, the principal symbol of $c_2$ has the form
\[
c_2^p(h,x,\xi)=c_1^p\Big(\frac{g+h}2,x,\xi\Big)\Big|_{g=0}=\mu\Big(\frac h2\Big)|\xi|^{-m}\theta(x),
\]
and the residual $r_2=c_2-c_2^p\in S^{-m-1}$.
 
Combining the above equations leads to
\begin{eqnarray*}
{\mathbb{E}}|u(x)|^2&=&\frac1{4|\kappa|^2}\int_{\mathbb R}\int_{\mathbb R
}e^{\kappa_{\rm i}h-2\kappa_{\rm i}x-{\rm i}\kappa_{\rm r}g}C_2(g,h,x)dgdh\\
&=&\frac1{4|\kappa|^2}\int_{\mathbb R}\int_{\mathbb R
}e^{\kappa_{\rm i}h-2\kappa_{\rm i}x-{\rm i}\kappa_{\rm r}g}\left[\frac1{4\pi}\int_{\mathbb R}e^{{\rm i}g\xi}c_2\Big(h,x,\xi\Big)d\xi\right]dgdh\\
&=&\frac1{8|\kappa|^2}\int_{\mathbb{R}}\int_{\mathbb{R}}e^{\kappa_{\rm i}h-2\kappa_{\rm i}x}\left[\frac1{2\pi}\int_{\mathbb{R}}e^{{\rm i} g(
\xi-\kappa_{\rm r})}dg\right]c_2(h,x,\xi) d\xi dh\\
&=&\frac1{8|\kappa|^2}\int_{\mathbb{R}}e^{\kappa_{\rm i}h-2\kappa_{\rm i}x}c_2(h,x,\kappa_{\rm r})dh\\
&=&\frac1{8|\kappa|^2}\int_{\mathbb{R}}e^{\kappa_{\rm i}h-2\kappa_{\rm i}x}\bigg[\mu\Big(\frac h2\Big)\kappa_{\rm r}^{-m}\theta(x)+r_2(h,x,\kappa_{\rm r})\bigg]dh\\
&=&\frac{\theta(x)}{4|\kappa|^2\kappa_{\rm r}^m}\int_{\mathbb{R}}e^{2\kappa_{\rm i}(\zeta-x)}\mu(\zeta)d\zeta+O(\kappa_{\rm r}^{-m-1}|\kappa|^{-2}).
\end{eqnarray*}
Finally, for any $x\in U$, we have from (\ref{eq:kappa}) that 
\[
\lim_{k\to\infty}4k^{m+2}{\mathbb{E}}|u(x)|^2=\lim_{k\to\infty}\left(\frac{k^{m+2}}{|\kappa|^2\kappa_{\rm r}^m}\right)
\int_{D}e^{2\kappa_{\rm i}(\zeta-x)}\mu(\zeta)d\zeta=T(x).
\]

On the other hand, if $x<y$ for any $x\in U$ and $y\in D$, we may repeat the
same procedure as above and show that
\[
\lim_{k\to\infty}4k^{m+2}{\mathbb{E}}|u(x)|^2=\lim_{k\to\infty}\left(\frac{k^{m+2}}{|\kappa|^2\kappa_{\rm r}^m}\right)
\int_{D}e^{2\kappa_{\rm i}(x-\zeta)}\mu(\zeta)d\zeta=T(x), 
\] 
which completes the proof.
\end{proof}

Now we are in the position to show that the strength $\mu$ of the covariance operator the random source is uniquely determined by the integral given in Theorem \ref{tm:inverse}.

\begin{theorem}\label{tm:unique}
Let $\sigma>0$. The strength $\mu$ is uniquely determined by
\[
T(x)=\int_{D}e^{-\sigma|x-y|}\mu(y)dy,\quad x\in U,
\]
where $U\subset\mathbb R\backslash\overline{D}$ is a bounded interval containing points from both sides of the interval $D$.
\end{theorem}

\begin{proof}
Let $g(x):=e^{-\sigma|x|}$. Then
\begin{equation}\label{Tx}
T(x)=(g*\mu)(x),\quad x\in U,
\end{equation}
and $T=g*\mu$ is a real analytic function. Hence, the value of $T$ can be
obtained everywhere according to the analytic continuation. 
Taking the Fourier transform of (\ref{Tx}) yields
\[
\mathcal{F}[\mu](\xi)=\frac{\mathcal{F}[T](\xi)}{\mathcal{F}[g](\xi)}=\frac{\sigma^2+\xi^2}{2\sigma}\mathcal{F}[T](\xi),
\]
which implies that $\mu$ can be uniquely determined by $T$. 
\end{proof}

\section{White noise}\label{wn}

In this section, we study the inverse random source problem where the source
is driven by a white noise. Specifically, we consider a centered random source
given in the form
\[
f=\sqrt{\mu}\dot{W},
\]
where $\dot{W}$ is the real-valued spatial white noise. The diffusion function
$\sqrt{\mu}$ is assumed to be a smooth function compactly supported in the
interval $D:=(0,1)$. By Lemma \ref{lm:iso}, it holds $f\in
W^{-\frac12-\epsilon,p}(D)$ for any $\epsilon>0$ and $p\in(1,\infty)$, which has
the same regularity as the microlocally isotropic Gaussian random field with
$m=0$. Moreover, the covariance operator $Q_f$ of $f$ satisfies
\begin{eqnarray*}
\langle\varphi,Q_f\psi\rangle&=&\mathbb E[\langle\sqrt{\mu}\dot{W},\varphi\rangle\langle\sqrt{\mu}\dot{W},\psi\rangle]\\
&=&\int_0^1\mu(y)\varphi(y)\psi(y)dy,
\end{eqnarray*}
which implies that
\[
K_f(x,y)=\mu(y)\delta(x-y)=\frac1{2\pi}\int_{\mathbb R}e^{{\rm i}(x-y)\xi}\mu(x)d\xi
\]
and hence the symbol of $Q_f$ is $c(x,\xi)=\mu(x)$ according to (\ref{eq:Kf}). As a result, $f=\sqrt{\mu}\dot{W}$ satisfies Assumption \ref{as:f} with $m=0$.

In this case, the solution $u$ of (\ref{eq:model}) is expressed by
\begin{equation}\label{eq:uwhite}
u(x)=(H_\kappa f)(x)=\frac1{2{\rm i}\kappa}\int_0^1e^{{\rm i}\kappa|x-y|}\sqrt{\mu(y)}dW(y).
\end{equation}
By It\^o's formula, we get
\begin{equation*}
\mathbb{E}|u(x)|^2=\frac1{4|\kappa|^2}\int_0^1e^{-2\kappa_{\rm i}|x-y|}\mu(y)dy,
\end{equation*}
which implies the uniqueness of determining the strength $\mu$ by following 
the same procedure as that in the proof of Theorem \ref{tm:unique}.

\begin{corollary}\label{coro:unique}
Let $D:=(0,1)$. If the random source $f$ has the form $f=\sqrt{\mu}\dot{W}$ with strength $\mu\in C_0^\infty(D)$ and $\mu\ge0$, then the strength $\mu$ can be uniquely determined by the following data at any fixed wave number $k$:
\begin{equation}\label{eq:formula}
4|\kappa|^2\mathbb E|u(x)|^2=\int_0^1e^{-2\kappa_{\rm i}|x-y|}\mu(y)dy,\quad x\in U,
\end{equation}
where $U\subset\mathbb R\backslash\overline{D}$ is a bounded interval containing points from both sides of the interval $D$.
\end{corollary}

\section{Numerical experiments}\label{ne}

In this section, we present the algorithmic implementation for the direct and
inverse scattering problems where the source is driven by the white noise, and
show some numerical examples to demonstrate the validity and effectiveness
of the proposed method.

\subsection{The scattering data} 

The measurement interval is chosen as $U=[-1.2,-0.2]\cup[1.2,2.2]$ which
satisfies $U\subset\mathbb R\backslash\overline{D}$ with $D=(0,1)$.
The scattering data $u(x)$ for all $x\in U$ is obtained by using the integral
equation (\ref{eq:uwhite}). Numerically, we generate the synthetic data at
discrete points $\{x_m\}_{m=0}^{M+1}\subset U$ defined by 
\[
x_0=-1.2,\quad x_{\frac M2+2}=1.2,\quad x_{m+1}=x_m+\Delta x
\]
for $m=0,\cdots,\frac M2,\frac M2+2,\cdots,M$ with $M=200$ and $\Delta x=2/M$, 
and approximate $u(x_m)$ by
\[
u(x_m)\approx\frac1{2{\rm i}\kappa}\sum_{n=0}^{N-1}e^{{\rm i}\kappa|x_m-y_n|}\sqrt{\mu(y_n)}\delta_nW,
\]
where 
\[
y_0=0,\quad y_{n+1}=y_n+\Delta y,\quad \delta_nW=W(y_{n+1})-W(y_n)
\]
for $n=0,\cdots,N-1$ with $N=200$ and $\Delta y=1/N$. The increments $\delta_nW$ with $n=0,\cdots,N-1$ defined above are independent and identically distributed, and hence can be simulated by $\sqrt{\Delta y}\xi_n$, where $\xi_n\in N(0,1)$ are independent and identically distributed random variables obeying the standard normal distribution.

\subsection{Reconstruction formula}

According to Corollary \ref{coro:unique}, the micro-correlation strength $\mu$
can be uniquely recovered by the energy $\mathbb E|u(x)|^2$ for $x\in U$ at a
fixed wave number $k$. However, the kernel $e^{-2\kappa_{\rm i}|x-y|}$ in the
integral in (\ref{eq:formula}) decays exponentially, which makes it difficult to
recover the high frequency modes of the strength $\mu$ numerically. To overcome
this difficulty, we use the following modified data instead in the numerical
experiments. Moreover, the multi-frequency data is used to enhance the
stability of the numerical solution.

Rewrite (\ref{eq:uwhite}) as
\[
2\kappa{\rm i}u(x)=\int_0^1e^{{\rm i}\kappa|x-y|}\sqrt{\mu(y)}dW(y),
\]
which can be split into the real and imaginary parts 
\begin{eqnarray*}
\Re[2\kappa{\rm i}u(x)]&=&\int_0^1e^{-\kappa_{\rm i}|x-y|}\cos(\kappa_{\rm r}|x-y|)\sqrt{\mu(y)}dW(y),\\
\Im[2\kappa{\rm i}u(x)]&=&\int_0^1e^{-\kappa_{\rm i}|x-y|}\sin(\kappa_{\rm r}|x-y|)\sqrt{\mu(y)}dW(y).
\end{eqnarray*}
Define the modified data
\begin{equation}\label{eq:data}
\mathcal{M}(x,k):=\mathbb E(\Re[2\kappa{\rm i}u(x)])^2-\mathbb E(\Im[2\kappa{\rm i}u(x)])^2.
\end{equation}
It can be verified that 
\begin{eqnarray*}
\mathcal{M}(x,k)&=&\int_0^1e^{-2\kappa_{\rm i}|x-y|}\cos^2(\kappa_{\rm r}|x-y|)\mu(y)dy\\
& & -\int_0^1e^{-2\kappa_{\rm i}|x-y|}\sin^2(\kappa_{\rm r}|x-y|)\mu(y)dy\\
&=&\int_0^1e^{-2\kappa_{\rm i}|x-y|}\cos(2\kappa_{\rm r}|x-y|)\mu(y)dy,
\end{eqnarray*}
whose evaluation at discrete points $\{x_m\}_{m=0}^{M+1}$ and wave number $k$ can be approximated by
\begin{equation}\label{eq:approx}
\mathcal{M}(x_m,k)\approx\Delta y\sum_{n=0}^{N-1}e^{-2\kappa_{\rm i}|x_m-y_n|}\cos(\kappa_{\rm r}|x_m-y_n|)\mu(y_n).
\end{equation}

The value of the strength $\mu$ at discrete points $\{y_n\}_{n=0}^{N-1}$ can be
numerically recovered by (\ref{eq:approx}) based on the truncated singular value
decomposition (SVD) with tolerance $\tau=10^{-3}$. Throughout the numerical
experiments, we use the average of $P=10^5$ sample paths as an approximation of
the expectation when calculating the data $\mathcal{M}$ in (\ref{eq:data}).

\subsection{Numerical examples}

We present three numerical examples to illustrate the performance of the
method. The first example contains only one Fourier mode and the second example
contains two Fourier modes. The third example contains more high Fourier modes
and the strength is more difficult to be recovered. 

\begin{figure}
\includegraphics[width=0.4\textwidth]{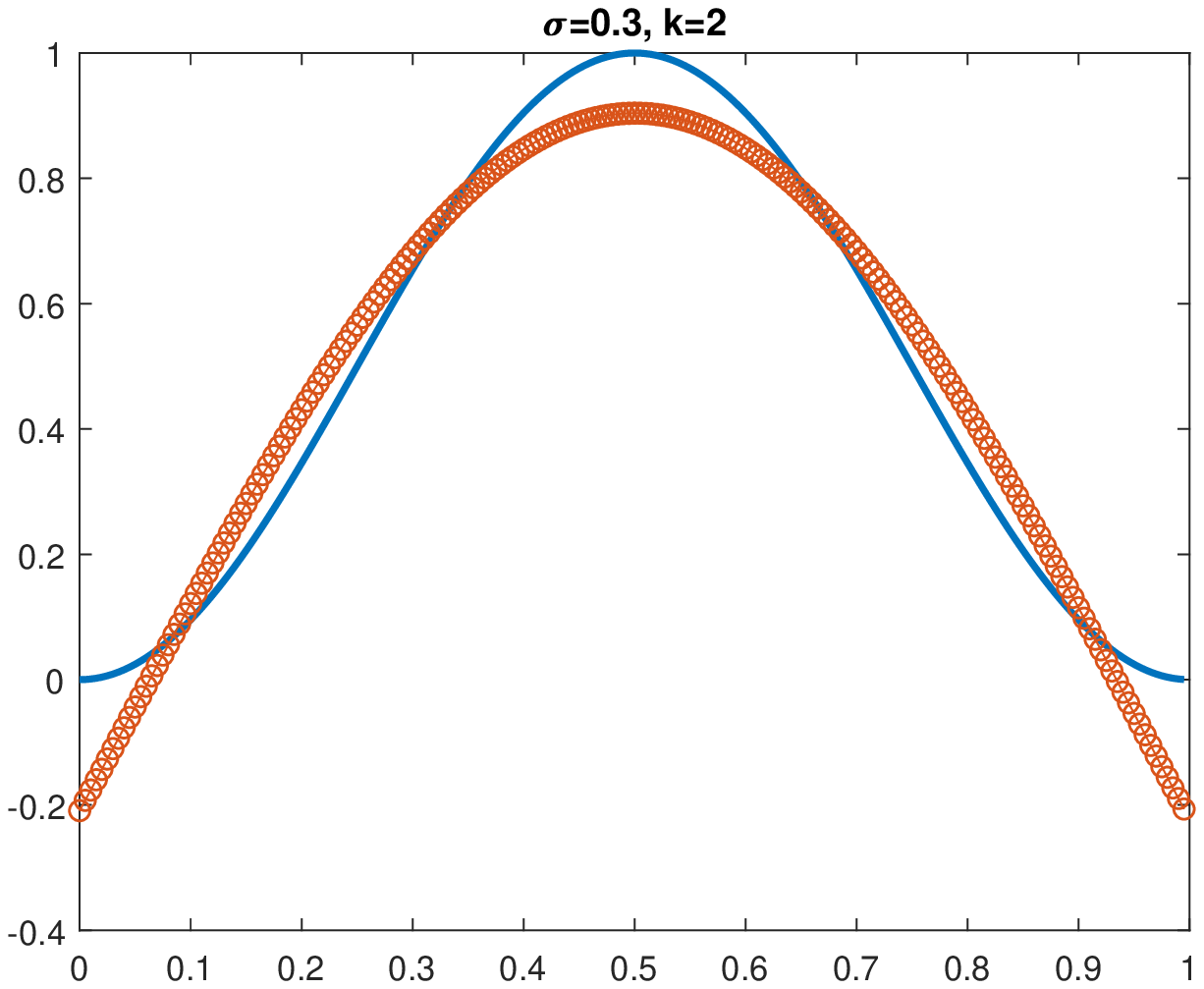}
\includegraphics[width=0.4\textwidth]{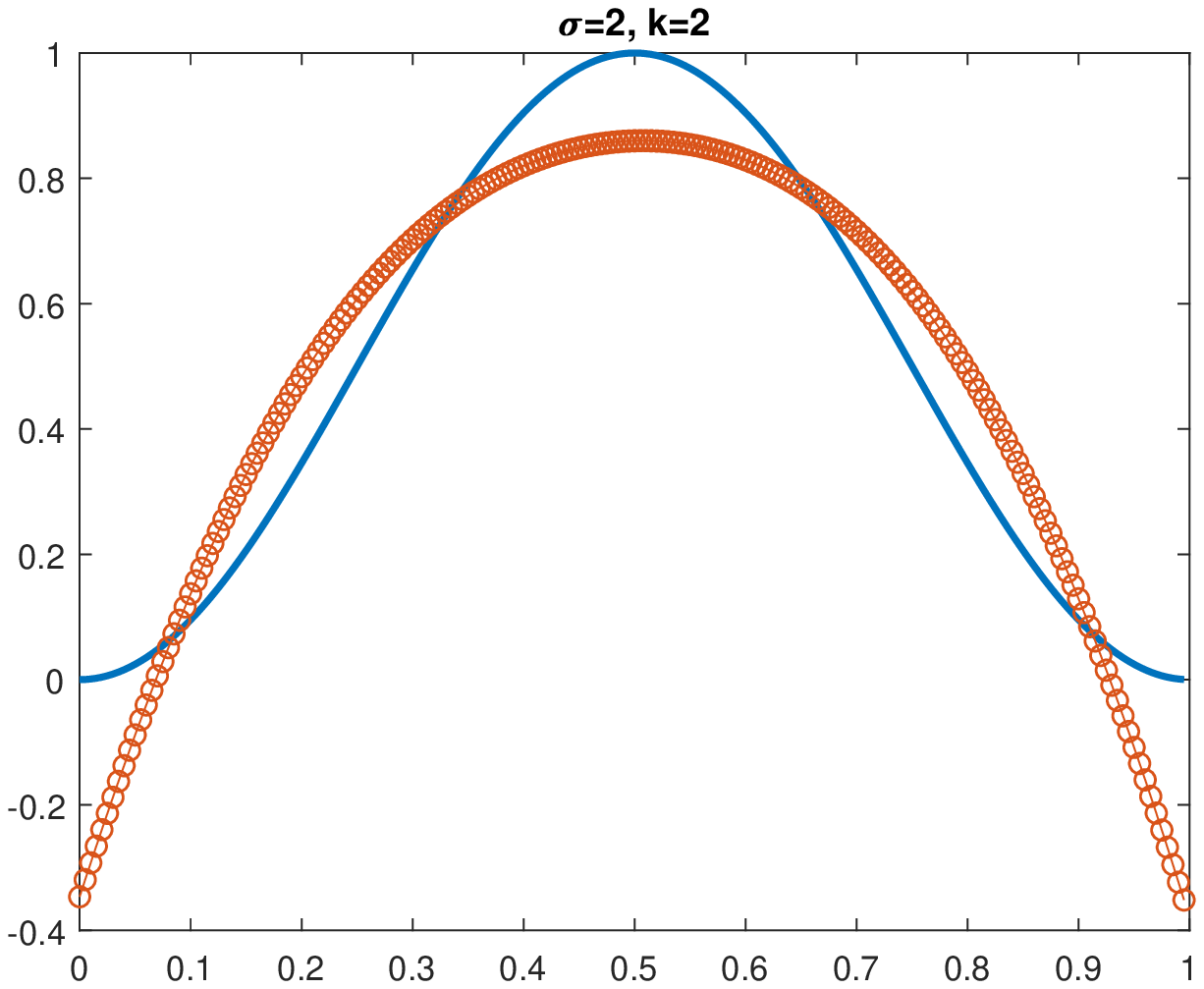}\\
\includegraphics[width=0.4\textwidth]{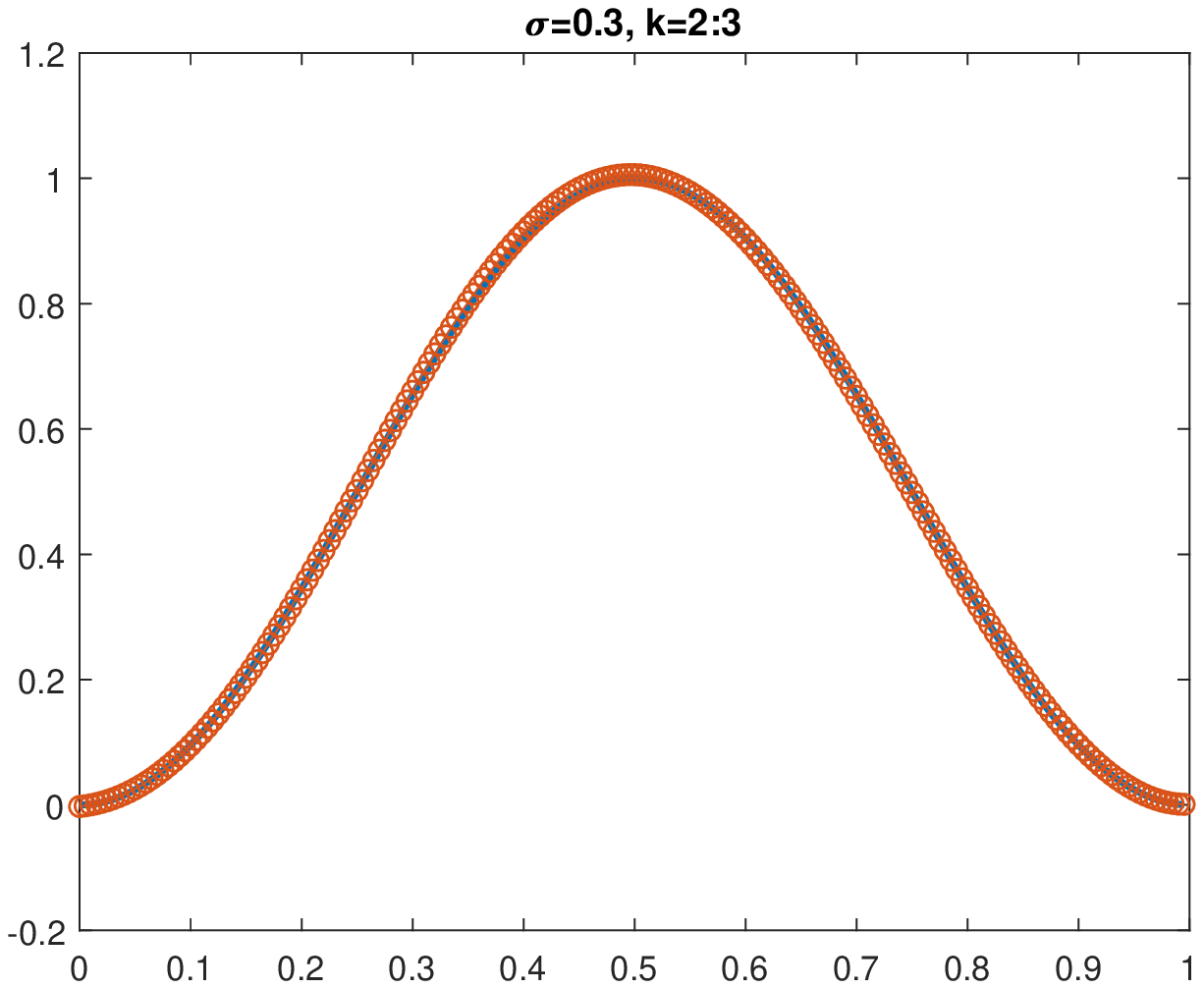}
\includegraphics[width=0.4\textwidth]{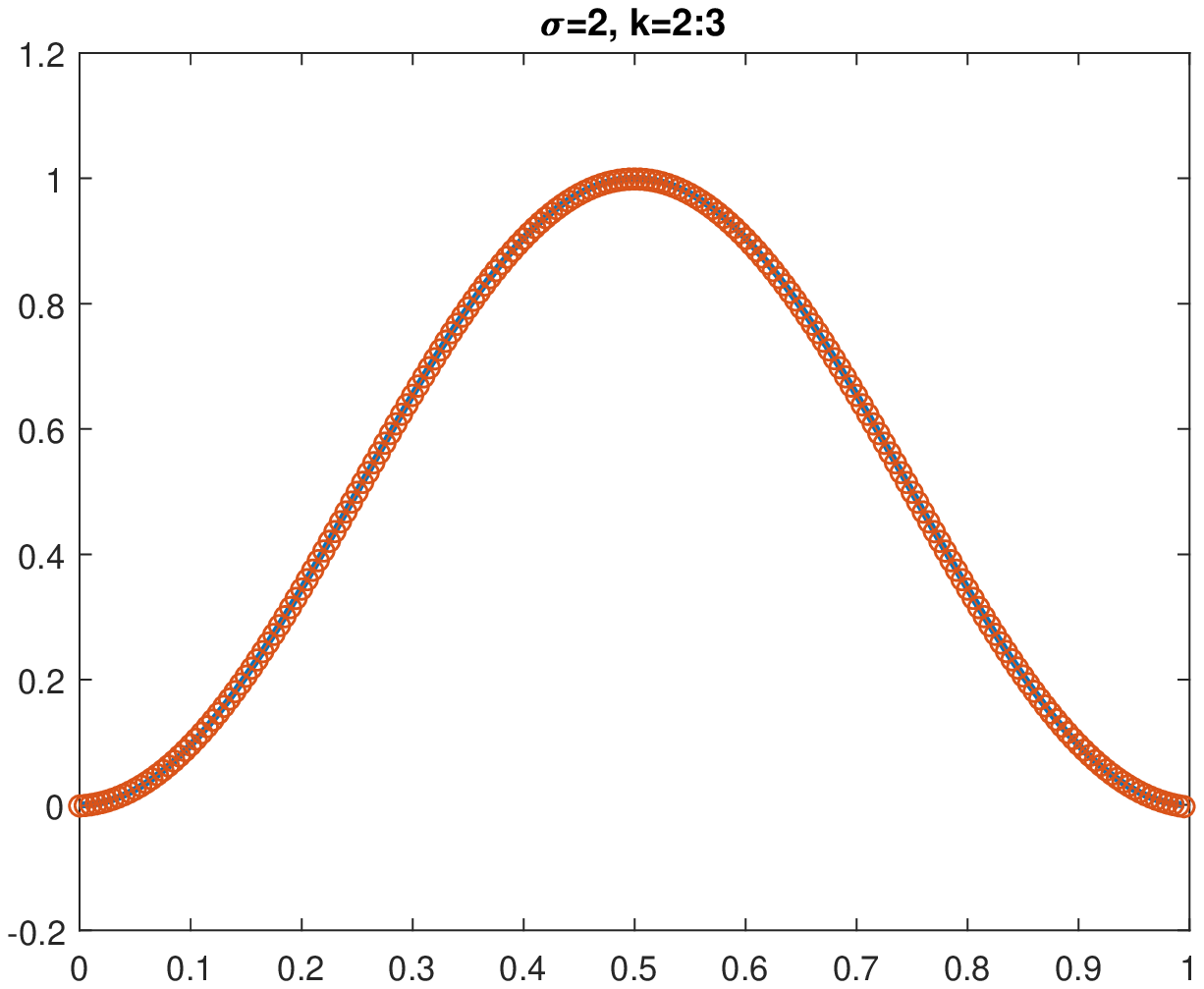}
\caption{Reconstruction of the strength in Example 1. Solid blue line:
exact strength; circled red line: reconstructed strength. For both $\sigma=0.3$
(left column) and $\sigma=2$ (right column), reconstructions based on data at
two frequencies $k=2,3$ are better than the ones based on data at a single
frequency $k=2$.}
\label{fig:1}
\end{figure}

\begin{figure}[h]
\includegraphics[width=0.3\textwidth]{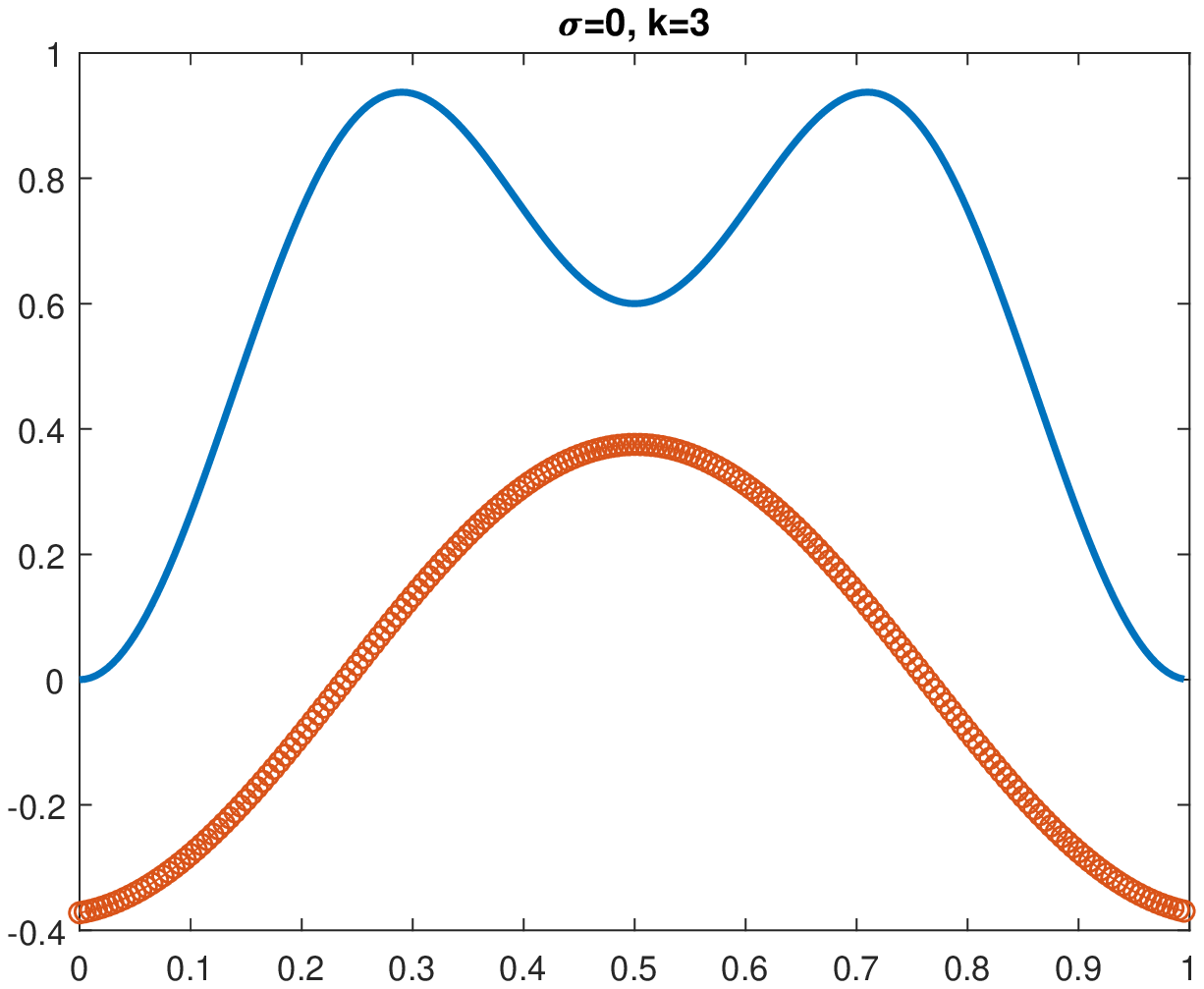}
\includegraphics[width=0.3\textwidth]{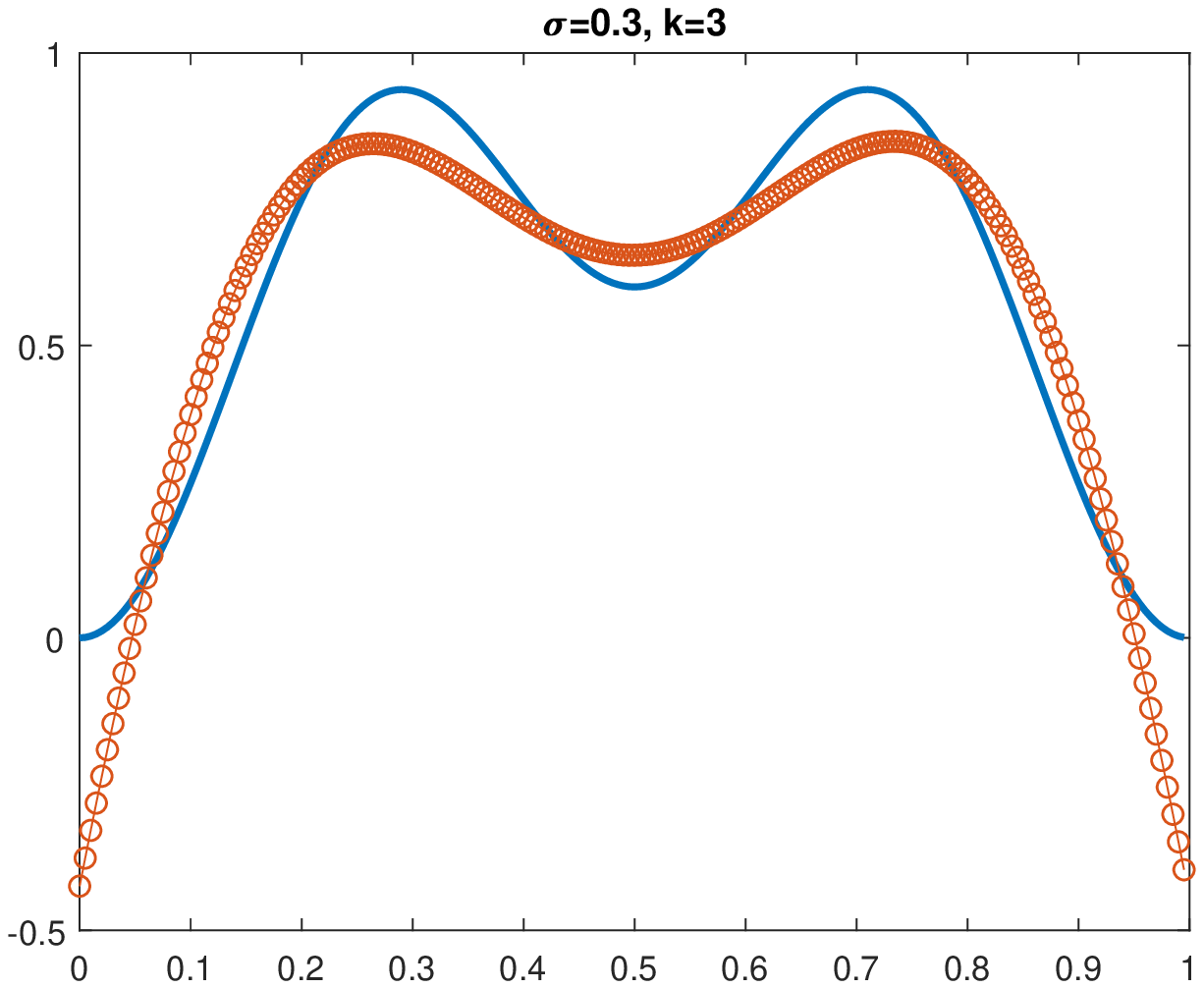}
\includegraphics[width=0.3\textwidth]{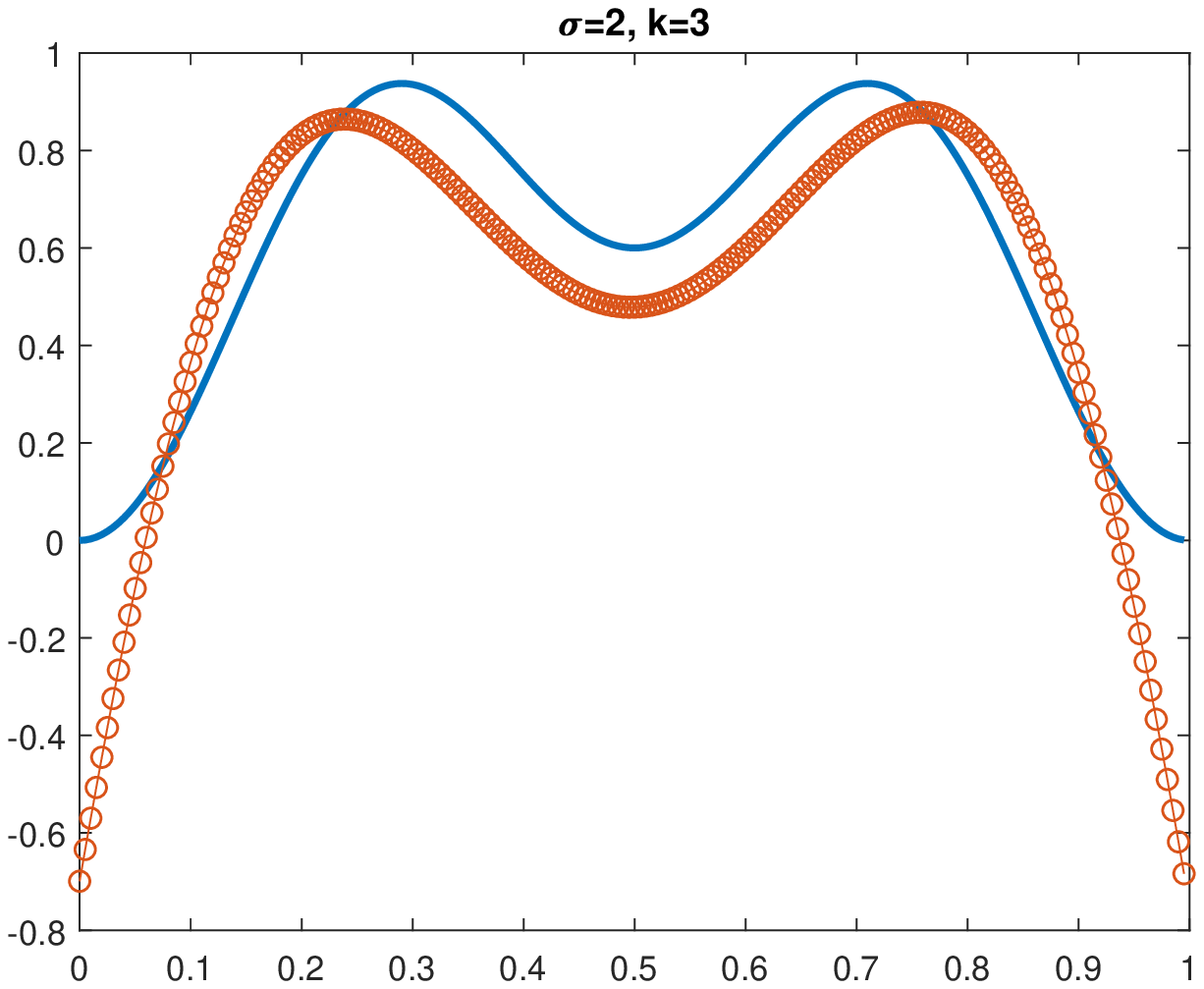}\\
\includegraphics[width=0.3\textwidth]{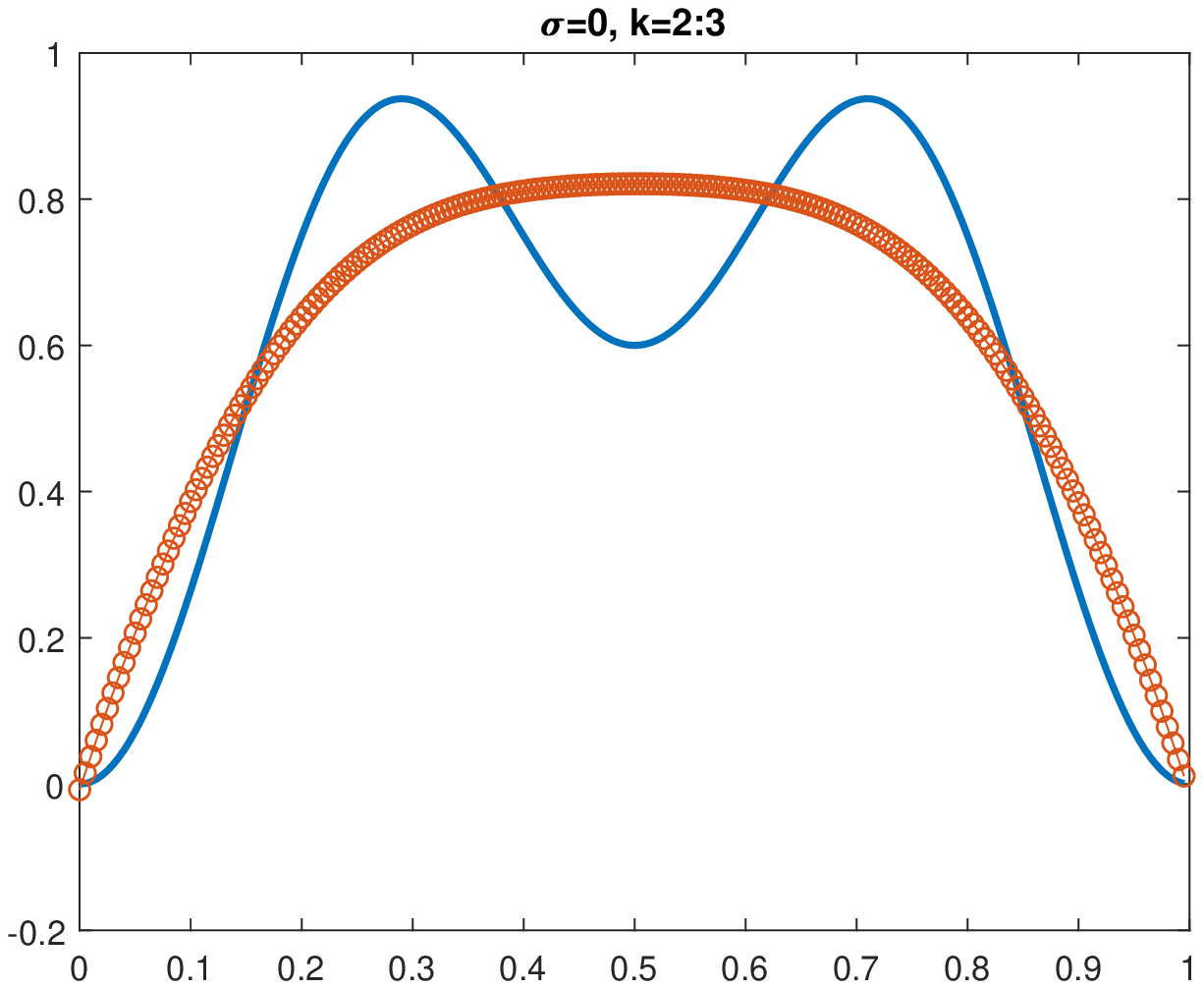}
\includegraphics[width=0.3\textwidth]{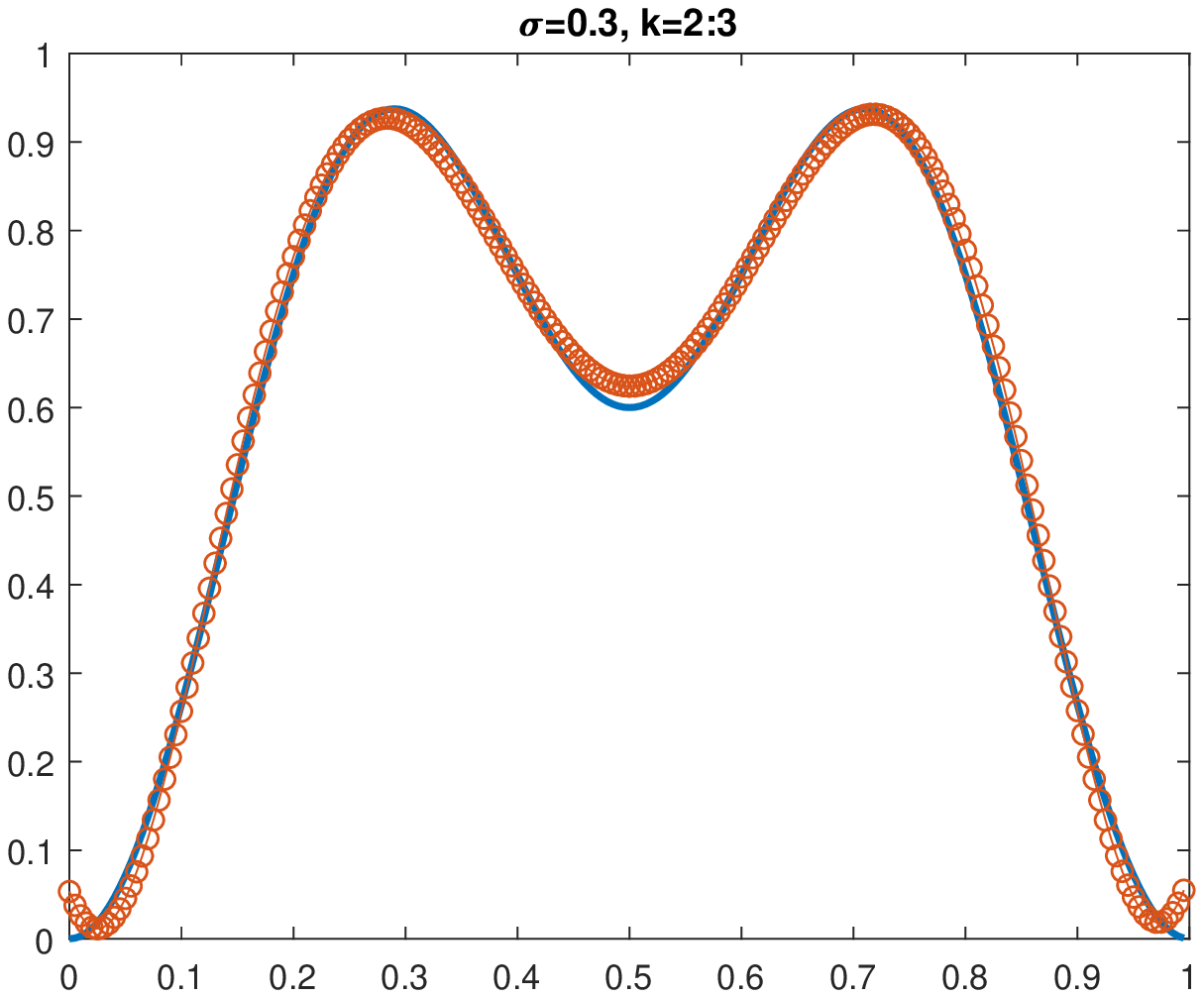}
\includegraphics[width=0.3\textwidth]{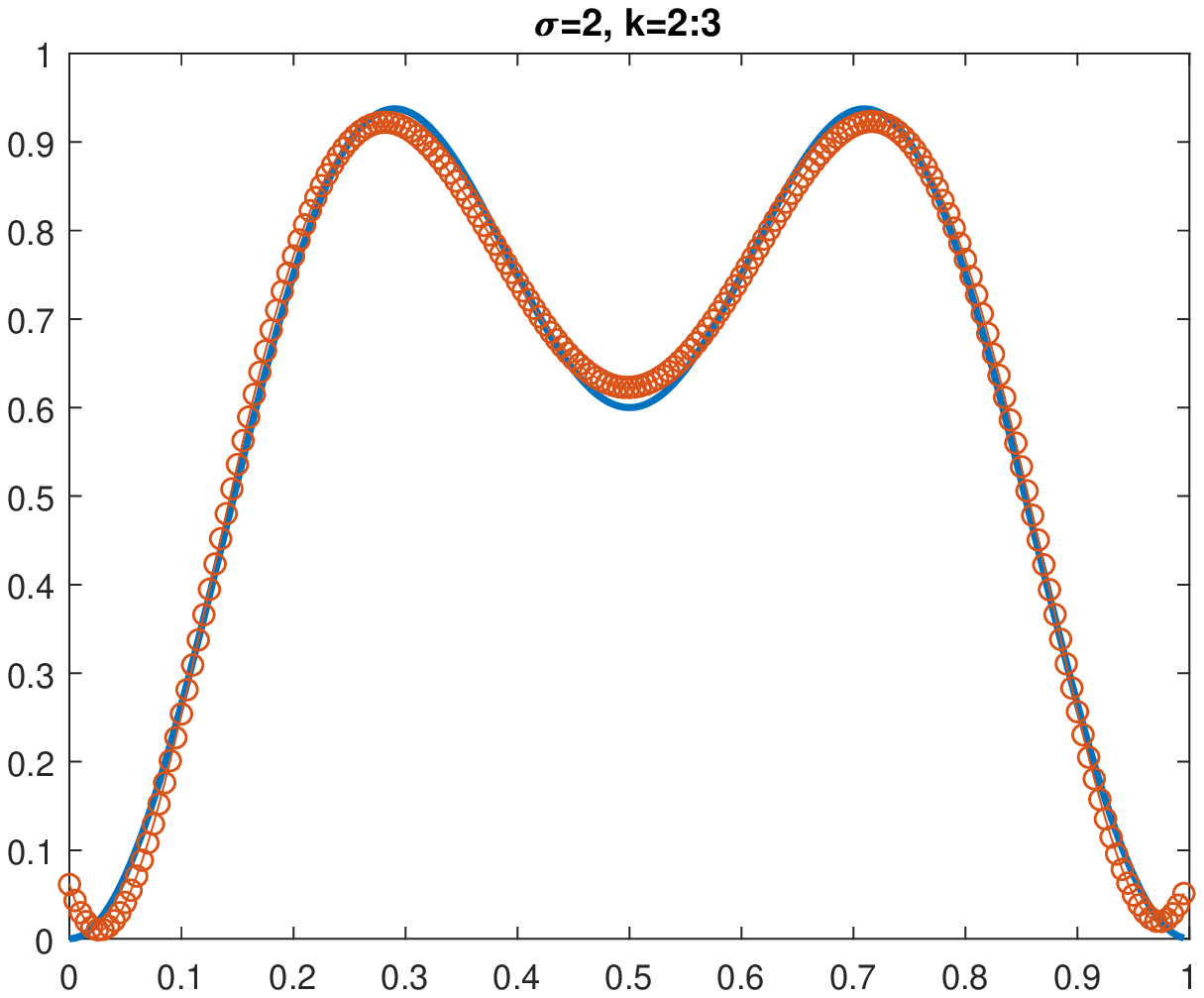}\\
\includegraphics[width=0.3\textwidth]{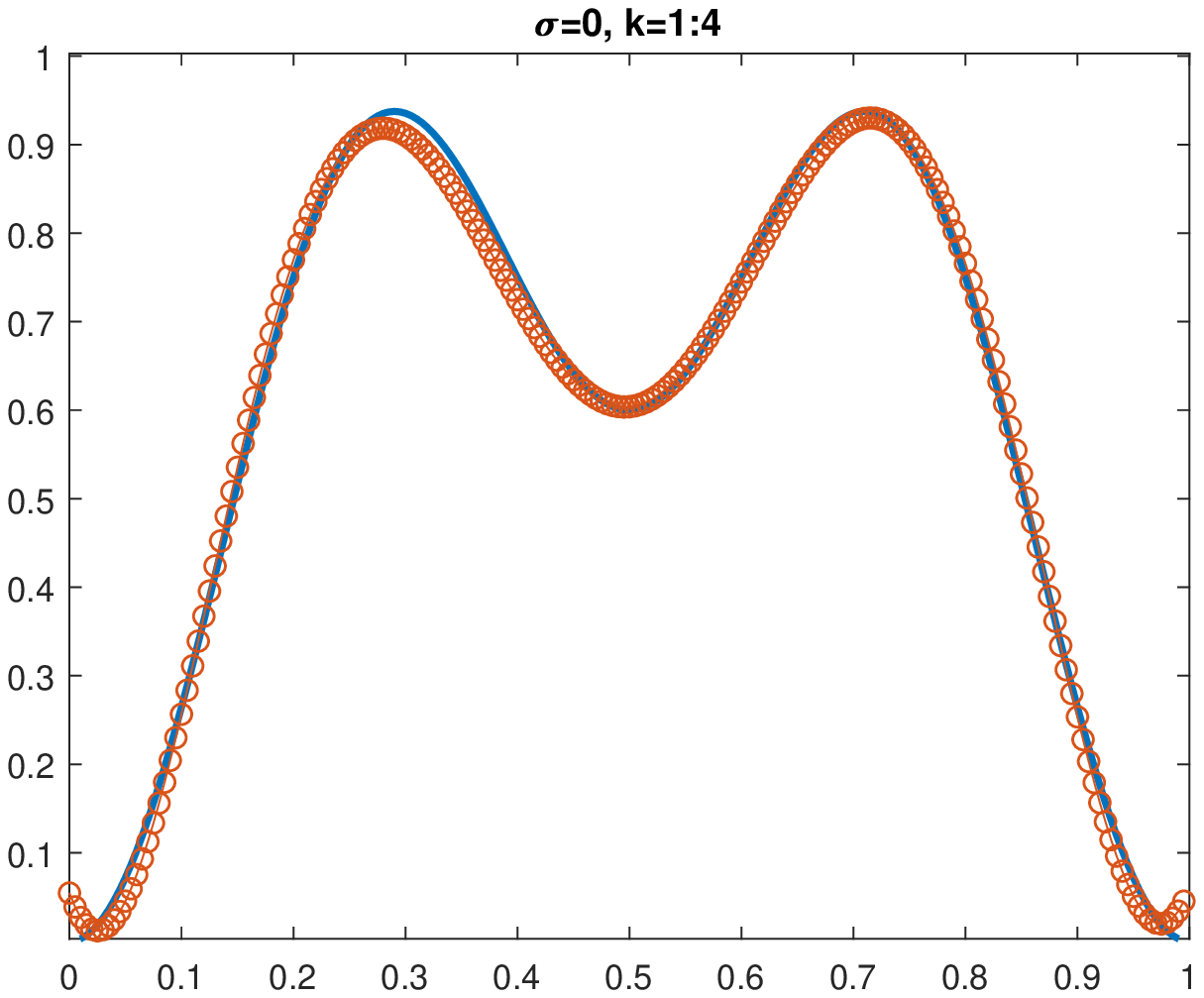}
\includegraphics[width=0.3\textwidth]{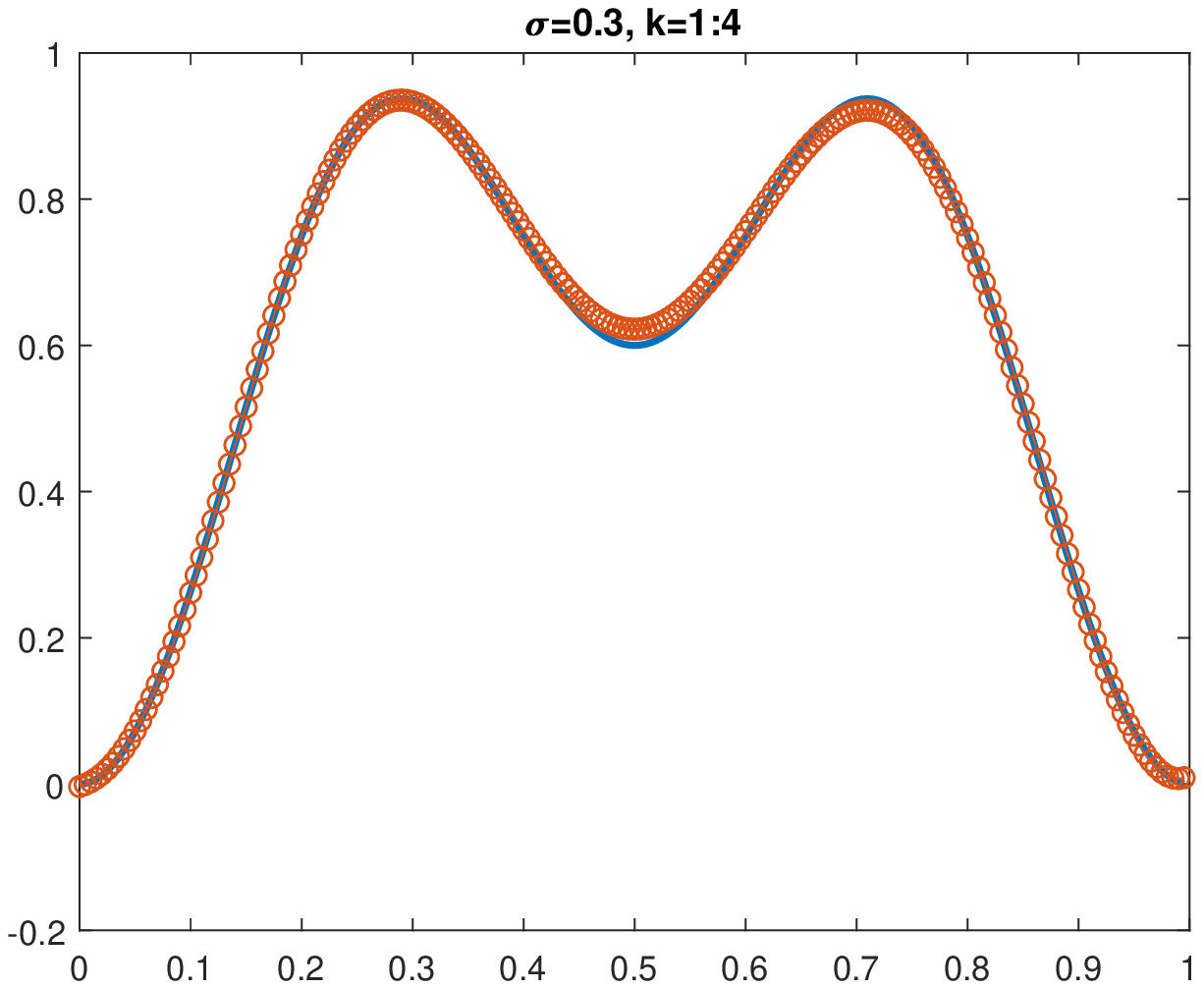}
\includegraphics[width=0.3\textwidth]{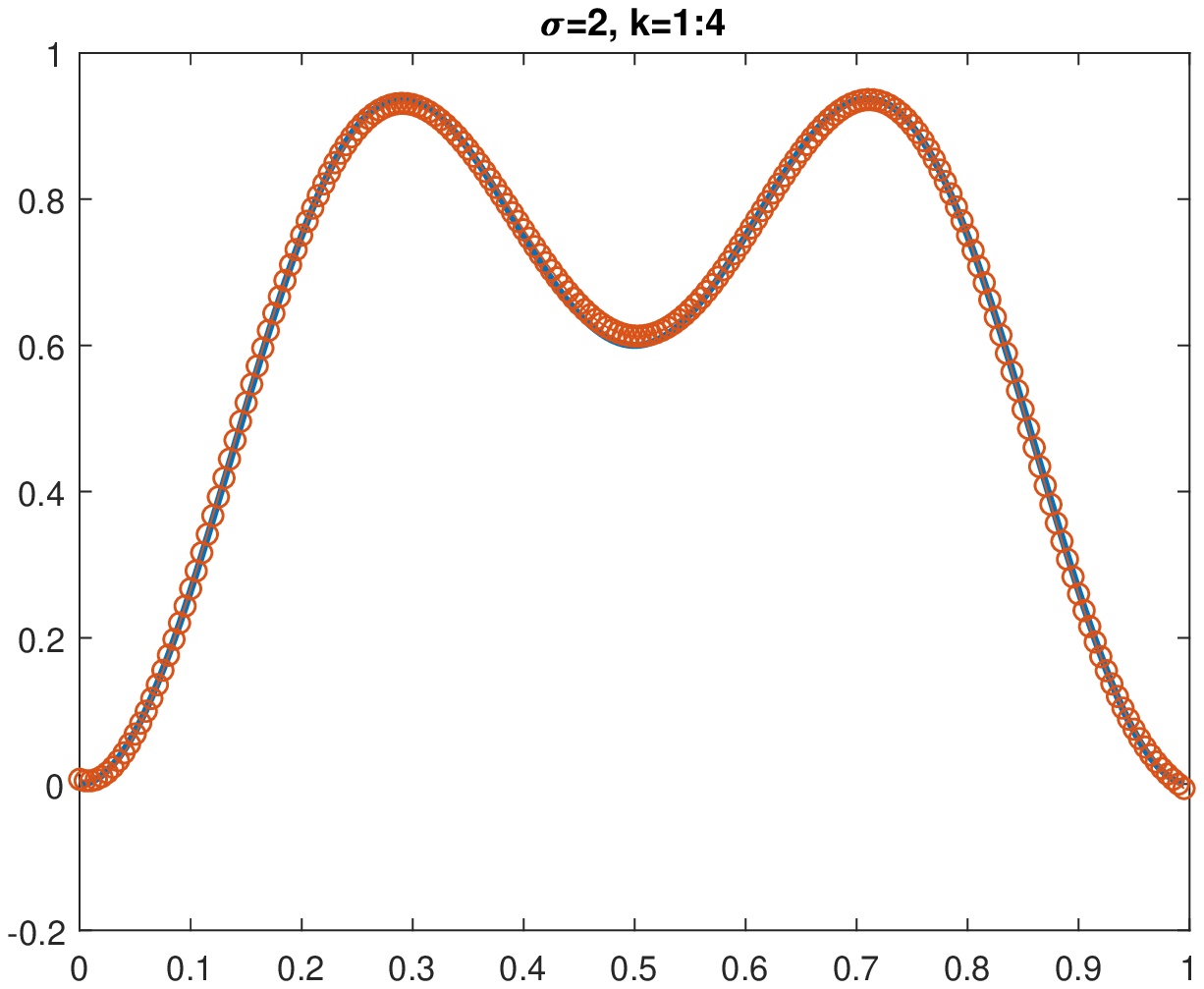}
\caption{Reconstruction of the strength in Example 2. Solid blue line:
exact strength; circled red line: reconstructed strength. For $\sigma=0$  (left
column), $\sigma=0.3$ (middle column) and $\sigma=2$ (right column),
reconstructions get better when data at more frequencies are used (top: $k=3$;
middle: $k=2,3$; bottom: $k=1:4$).}
\label{fig:2}
\end{figure}

\begin{example}
\normalfont Reconstruct the strength given by
\[
\mu(x)=0.5(1-\cos(2\pi x))
\]
inside the interval $(0,1)$. 
Figure \ref{fig:1} plots the reconstructed strength and the exact one based on
the modified data $\mathcal{M}(x_m,k)$ with different attenuation coefficients
$\sigma=0.3,2$ at one frequency $k=2$ and two frequencies $k=2,3$. As expected,
the better reconstruction can be obtained when data at more frequencies is
used. The strength $\mu$ can be properly recovered by data at two frequencies
$k=2,3$ since $\mu$ considered in this example contains one low frequency
Fourier mode. 
\end{example}

\begin{example}
\normalfont Reconstruct the strength given by
\[
\mu(x)=0.6-0.3\cos(2\pi x)-0.3\cos(4\pi x)
\]
inside the interval $(0,1)$. This example contains two Fourier modes and
is a little harder than Example 1. Figure \ref{fig:2} shows the reconstructed
strength and the exact one based on the modified data $\mathcal{M}(x_m,k)$ with
different attenuation coefficients $\sigma=0,0.3,2$ at one frequency $k=3$, two
frequencies $k=2,3$ and four frequencies $k=1,2,3,4$. Note that if $\sigma=0$,
then $\kappa_{\rm i}=0$ and the exponential kernel in (\ref{eq:formula})
vanishes. Hence, only the average of the strength $\mu$ can be recovered, and
the strength itself can not be uniquely determined based on the data at a single
frequency in this case. To reconstruct the strength $\mu$ for the case
$\sigma=0$, the multi-frequency data is required. We refer to \cite{BCLZ14,L11}
for the details of inverse random source problem of the one-dimensional
Helmholtz equation without attenuation. For $\sigma=0.3,2$, the strength $\mu$
can be properly recovered by using the data at a few frequencies. 
\end{example}

\begin{example}
\normalfont Reconstruct the strength given by
\[
\mu(x)=0.5e-0.3e^{\cos(4\pi x)}-0.2e^{\cos(6\pi x)}
\]
inside the interval $(0,1)$. The strength $\mu$ in this example contains more
higher Fourier modes than the two previous examples. Hence, it is expected
that the data at more frequencies is required to reconstruct the strength.
Figure \ref{fig:3} shows the reconstructed strength and the exact one based on
data with $\sigma=0,0.3,2$ at one frequency $k=3$ or more frequencies
$k=1,\cdots,8$ and $k=1,\cdots,16$, respectively. For $\sigma=0$, data at a
single frequency can hardly recover the strength. For $\sigma=0.3,2$, data at a
single frequency could roughly recover the strength, and the reconstructions get
better when data at more frequencies is used.
\end{example}

\begin{figure}[h]
\includegraphics[width=0.3\textwidth]{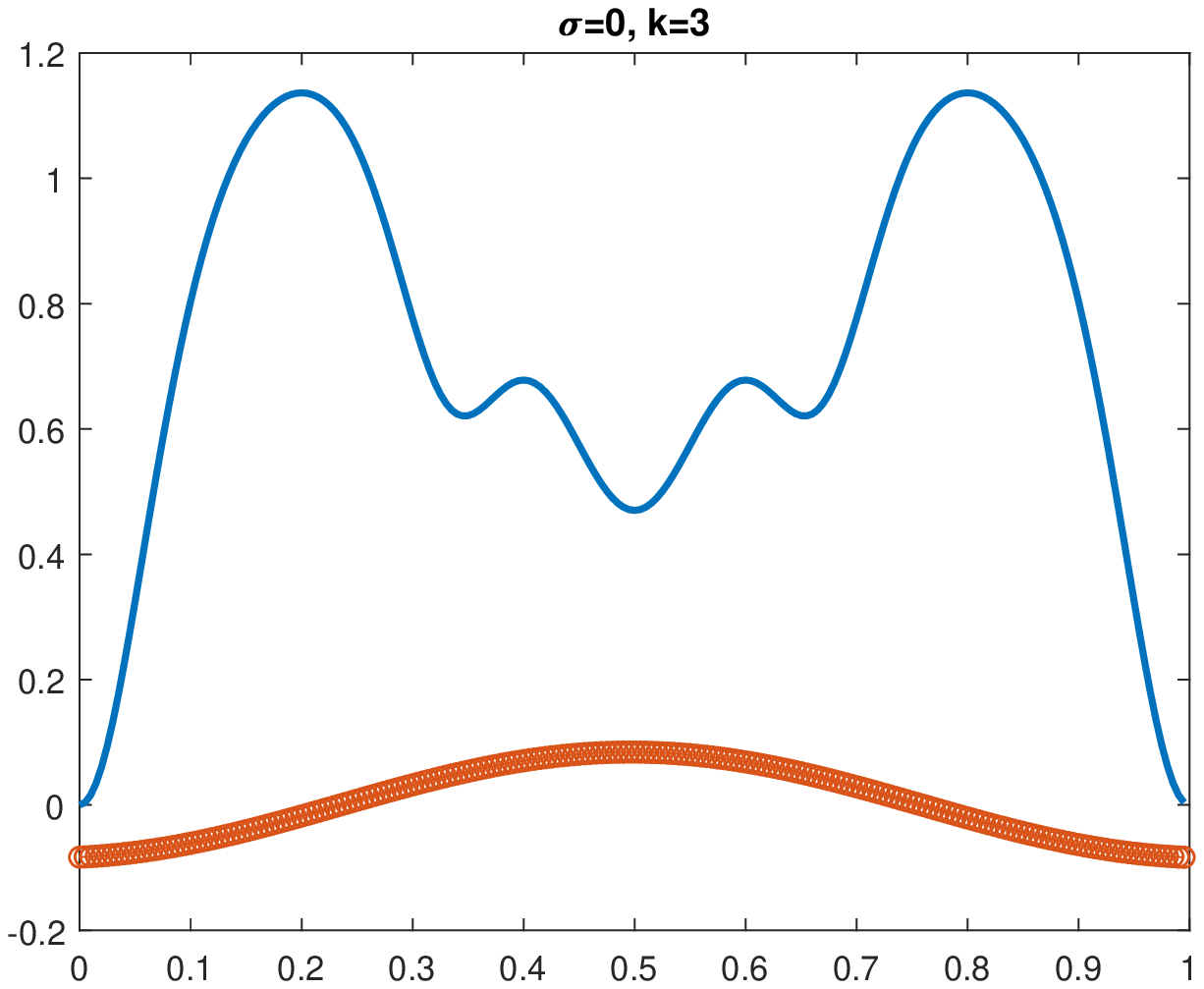}
\includegraphics[width=0.3\textwidth]{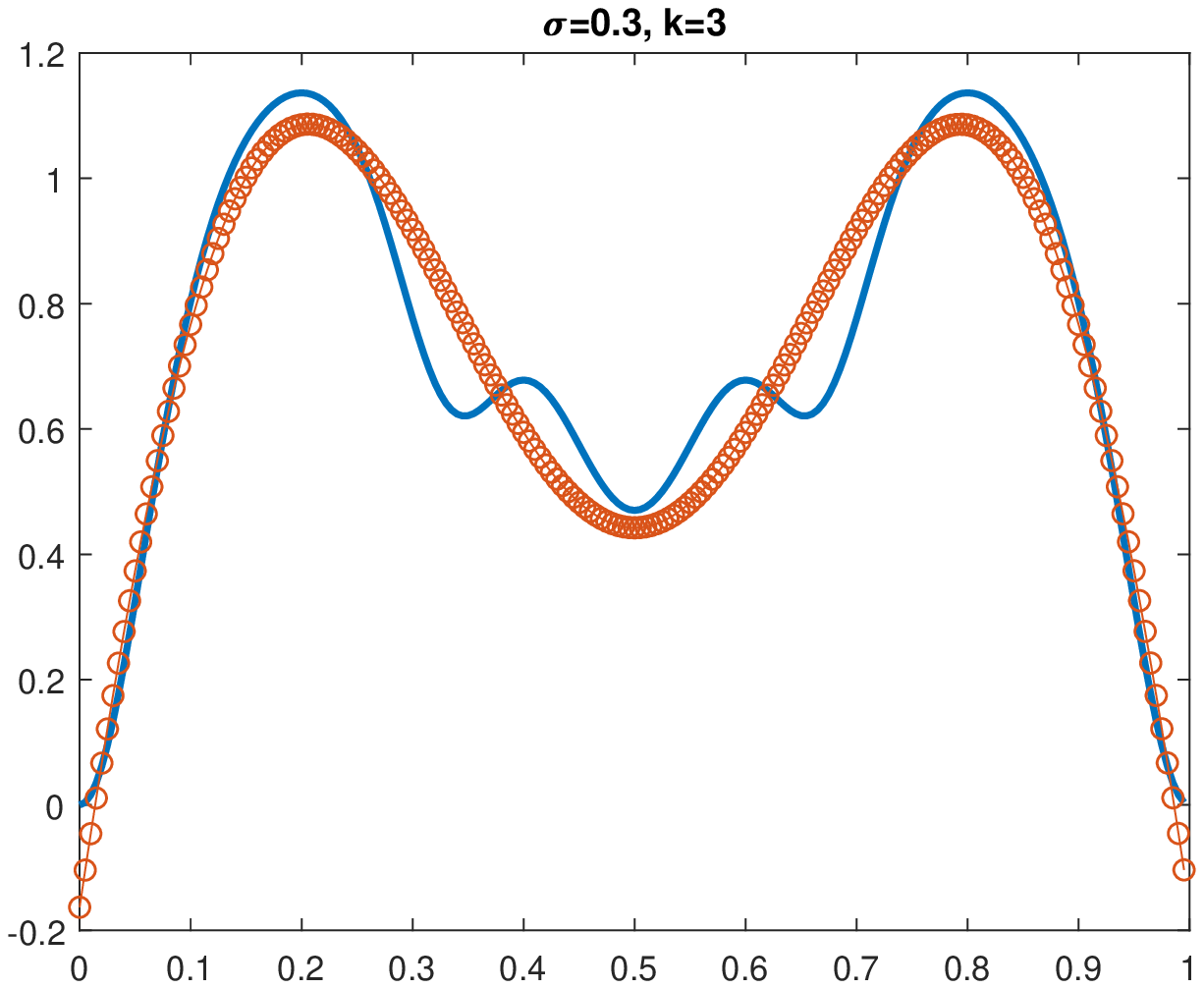}
\includegraphics[width=0.3\textwidth]{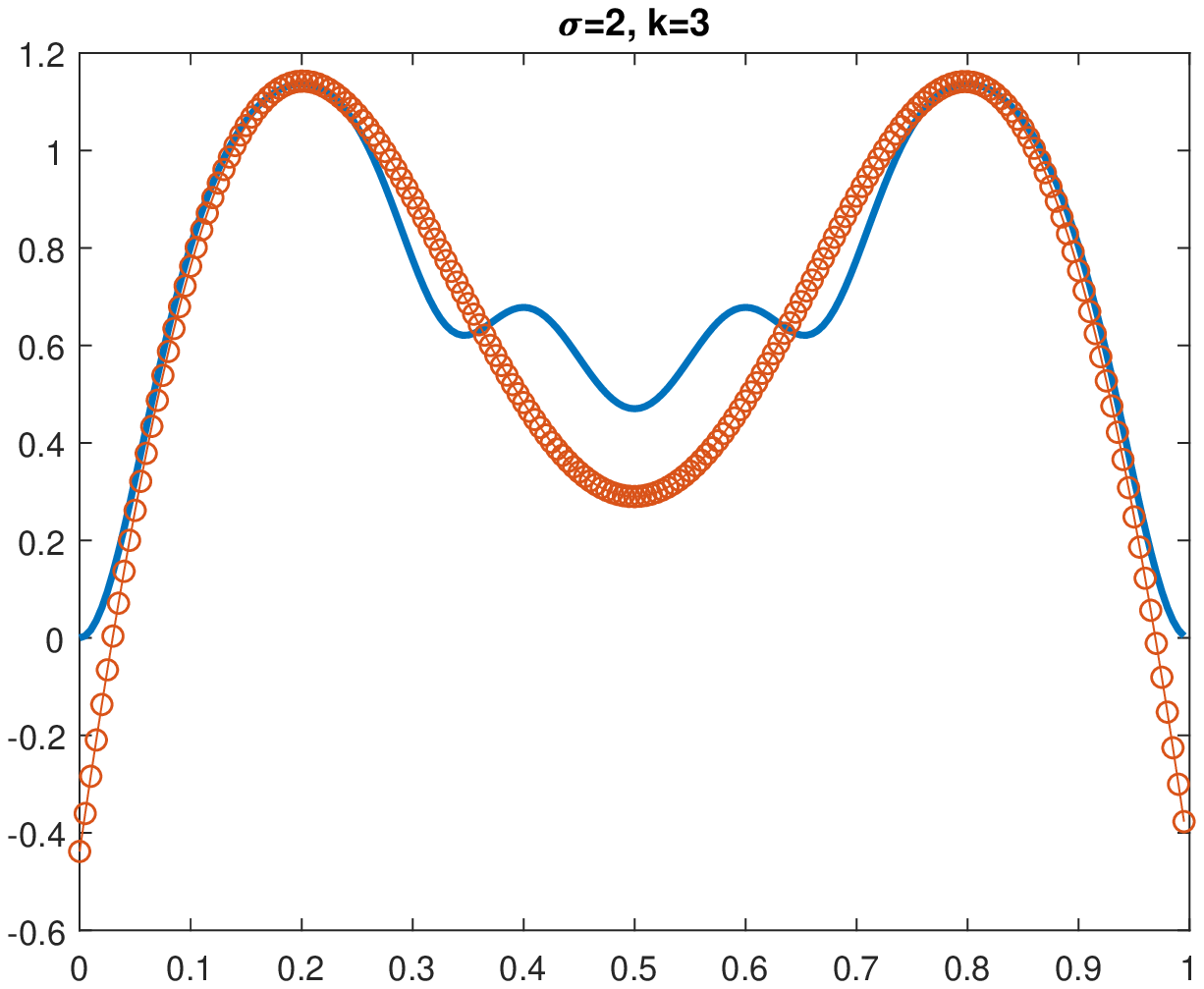}\\
\includegraphics[width=0.3\textwidth]{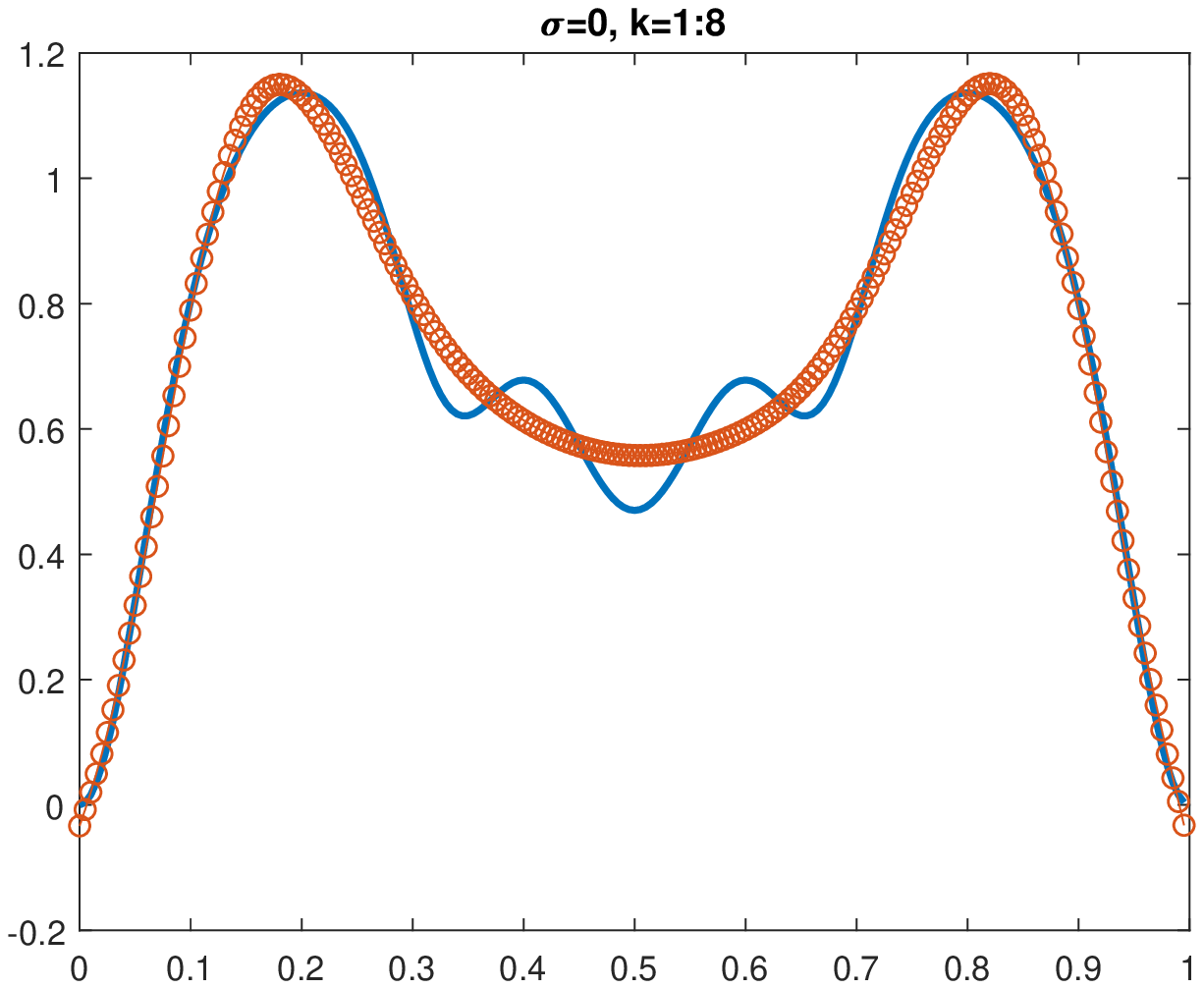}
\includegraphics[width=0.3\textwidth]{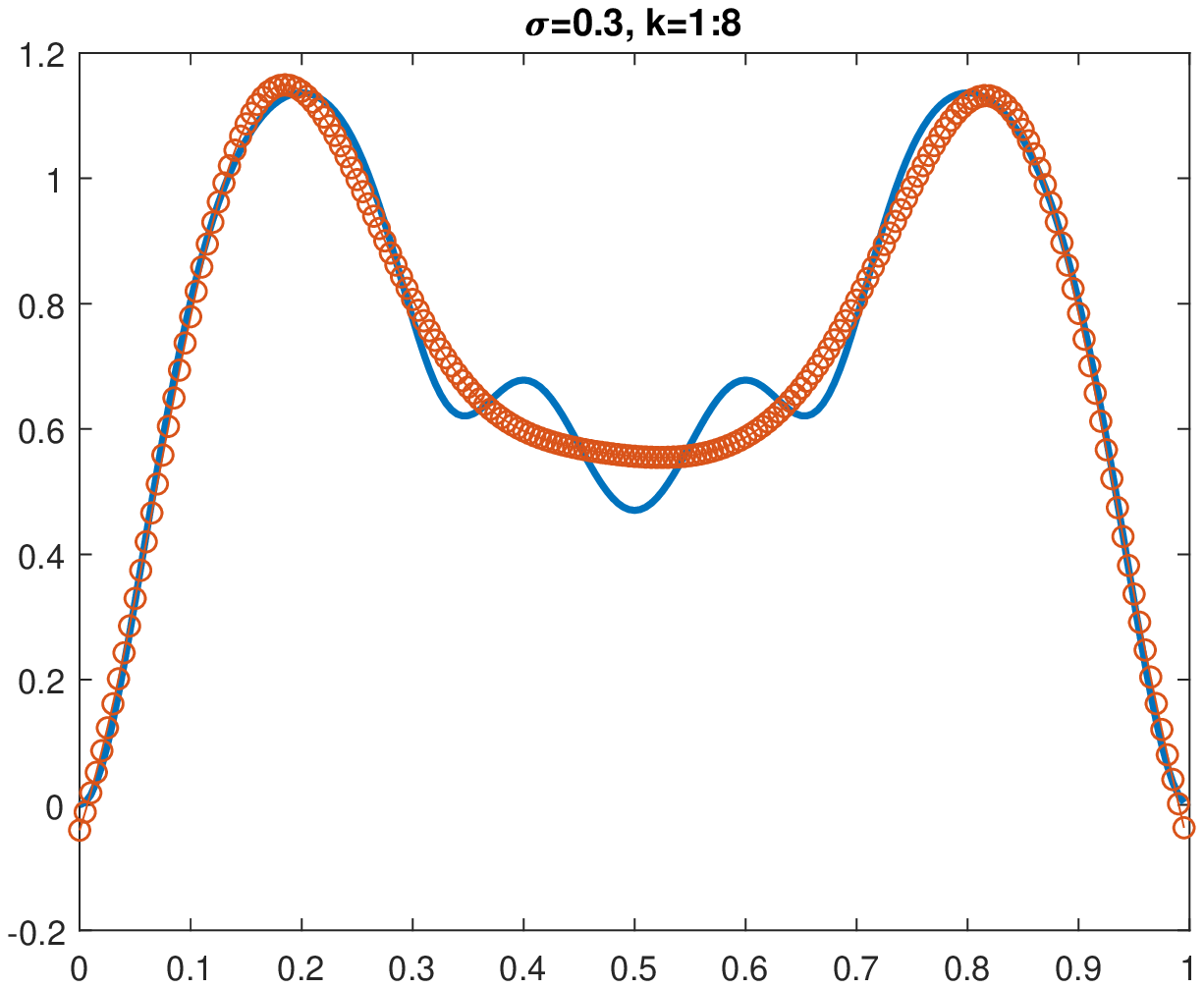}
\includegraphics[width=0.3\textwidth]{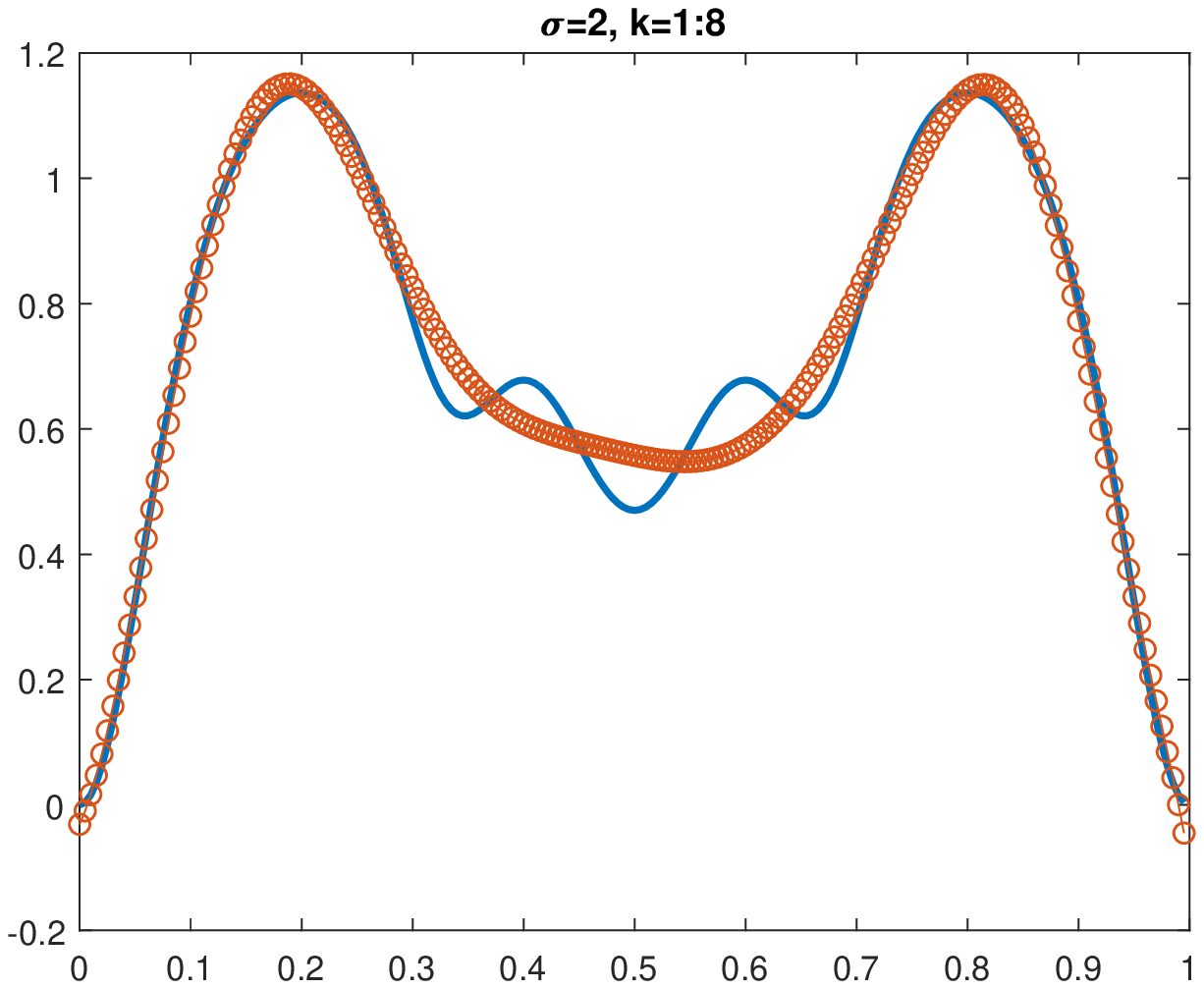}\\ 
\includegraphics[width=0.3\textwidth]{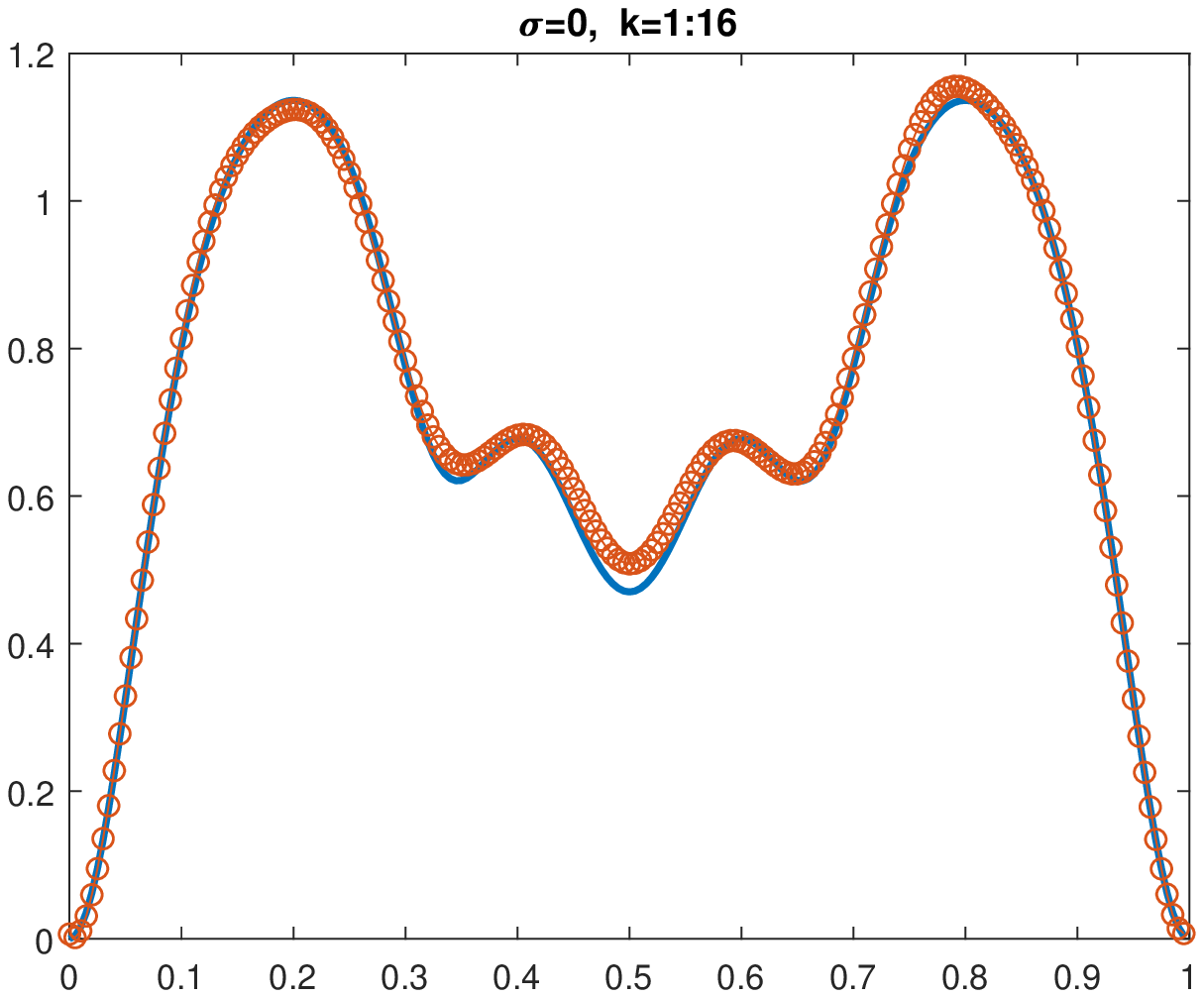}
\includegraphics[width=0.3\textwidth]{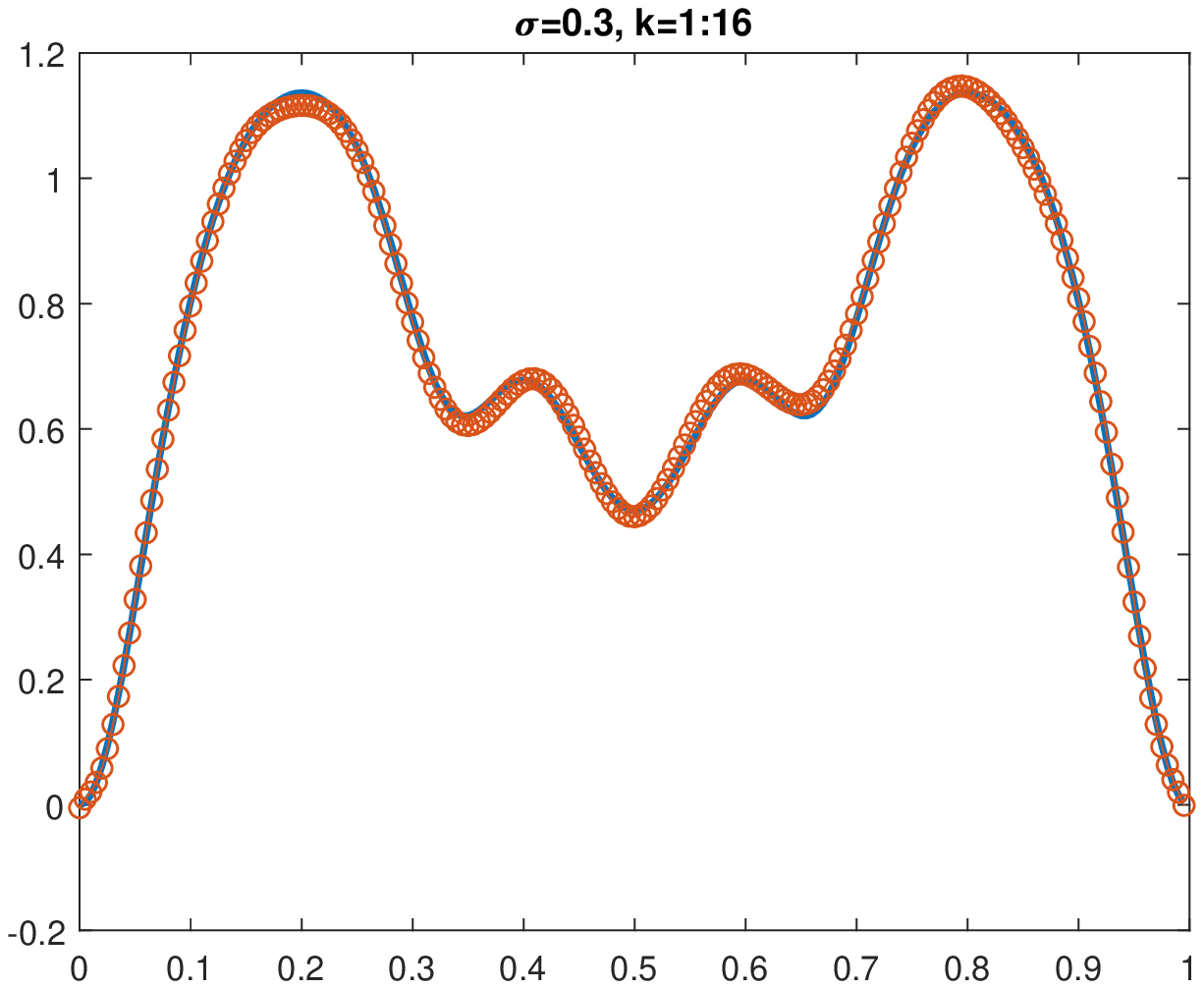}
\includegraphics[width=0.3\textwidth]{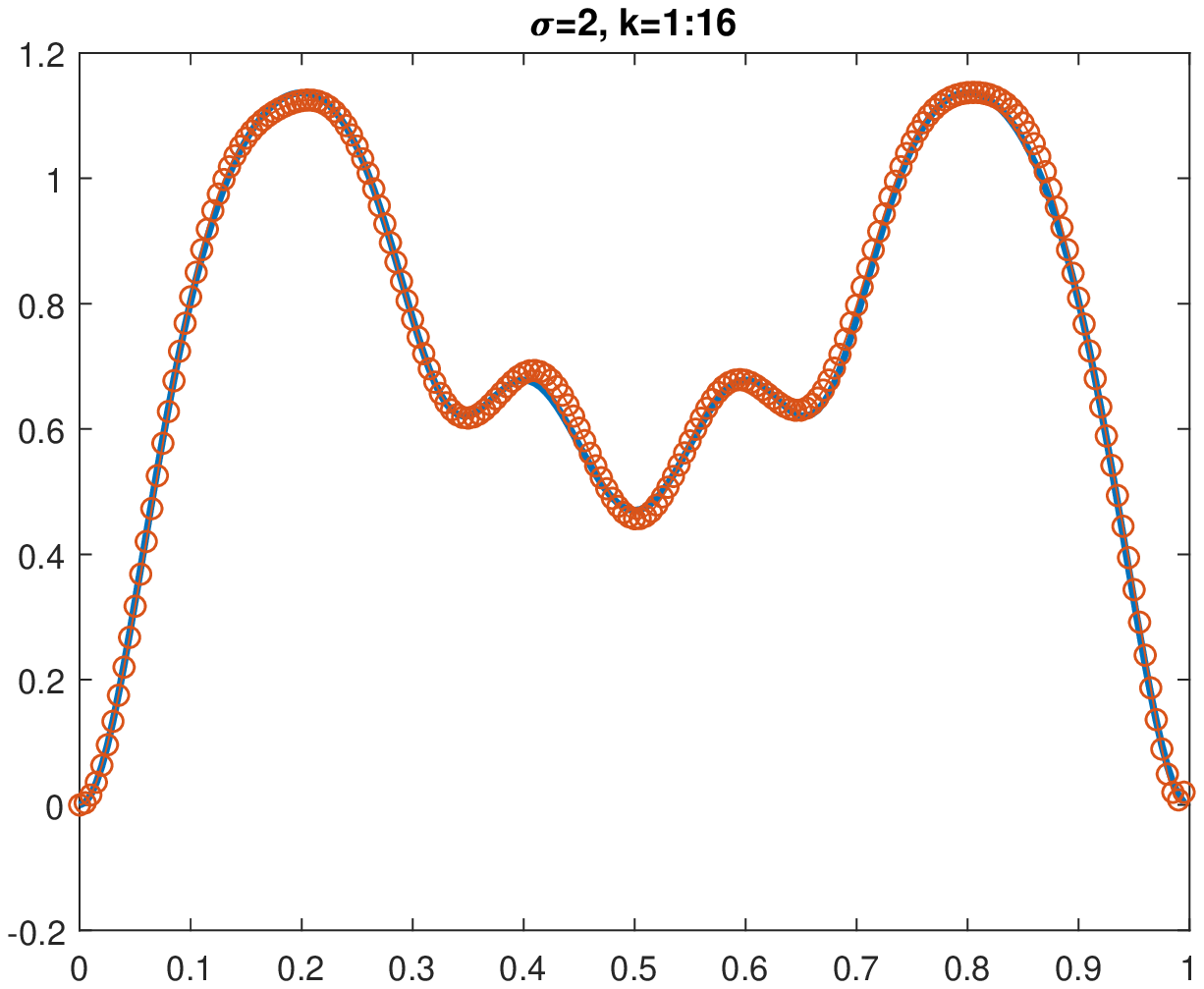}
\caption{Reconstruction of the strength in Example 3. Solid blue line:
exact strength; circled red line: reconstructed strength. For $\sigma=0$ (left
column), $\sigma=0.3$ (middle column) and $\sigma=2$ (right column),
reconstructions get better when data at more frequencies are used (top: $k=3$;
middle: $k=1,\cdots,8$; bottom: $k=1,\cdots,16$), and more frequencies are
required to properly reconstruct the strength which contains more higher
frequency modes.}
\label{fig:3}
\end{figure}

\section{Conclusion}\label{c}

We have studied an inverse random source scattering problem for
the one-dimensional Helmholtz equation with attenuation, which is to reconstruct
the strength of the random source. Compared with higher dimensional cases
studied in \cite{LW}, the fundamental solution in the one-dimensional case is
smooth, which makes it possible to deal with rougher random sources including
the white noise. The strength is shown to be uniquely determined by the variance
of the wave field in an open measurement set. The attenuation is essential in
the model to get the strength reconstructed point-wisely.

It is open for the recovery of microlocally isotropic random sources for
the one-dimensional Helmholtz equation without attenuation as well as the
recovery of  microlocally isotropic random media for the Helmholtz equation.
For the one-dimensional Helmholtz equation without attenuation, only the average
of the strength of the microlocally isotropic Gaussian random source
over its support could be obtained from on the method presented in this work.
For the higher dimensional Helmholtz equation in microlocally isotropic random
media, the well-posedness of the direct scattering problem has been studied in
\cite{LW3}; for the inverse problem, however, it is difficult to get the
convergence of the Born series generated by the Lippmann--Schwinger equation,
which makes it difficult to get an explicit expression of the strength of the
random media. Some other mathematical tools need to be explored to deal with
these open problems.

\end{document}